\documentclass[preprint,11pt,authoryear]{elsarticle}

\usepackage[margin=2cm]{geometry}
\usepackage[hidelinks,colorlinks = true,linkcolor={blue!80!black},citecolor={blue!50!black},urlcolor={blue!80!black}]{hyperref}
\usepackage{amsmath,amsfonts}
\usepackage{amsthm}
\usepackage{enumitem}% http://ctan.org/pkg/enumitem
\usepackage{amssymb}
\let\emptyset\varnothing
\usepackage{multirow}
\usepackage{float}
\usepackage{subcaption}
\usepackage{color}
\usepackage{adjustbox}
\usepackage{booktabs}
\usepackage{multirow}
\usepackage{soul}
\usepackage{algpseudocode}
\usepackage{algorithm}
\usepackage[usenames,dvipsnames]{xcolor}

\algnewcommand{\IfThenElse}[3]{% \IfThenElse{<if>}{<then>}{<else>}
  \State \algorithmicif\ #1\ \algorithmicthen\ #2\ \algorithmicelse\ #3}

\usepackage[colorinlistoftodos]{todonotes}

\usepackage[colorinlistoftodos]{todonotes}
\definecolor{cerulean}{rgb}{0.0, 0.40, 0.60}

\def\J{\mathcal{J}}
\def\I{\mathcal{I}}

\newtheorem{theorem}{Theorem}
\newtheorem{defn}{Definition}

\newtheorem{lemma}[theorem]{Lemma}
\newtheorem{prop}[theorem]{Proposition}
\newtheorem{example}{Example}
\newtheorem{remark}{Remark}
\allowdisplaybreaks

\begin{document}
\newlength{\medidaparentesis}
\settowidth{\medidaparentesis}{(00)}

\begin{frontmatter}
\title{Stable formulations for the Capacitated Facility Location Problem with Customer Preferences}

\author[um]{Concepci\'on Dom\'inguez$^*$}
\ead{concepcion.dominguez@um.es}
\author[um]{Juan de Dios Jaime-Alcántara}
\ead{jd.jaimealcantara@um.es}
\address[um]{Department of Statistics and Operations Research, University of Murcia, Murcia 30100, Spain }
\cortext[cor1]{Corresponding author.}

\allowdisplaybreaks

\begin{abstract} 
In the Simple Plant Location Problem with Order (SPLPO), the aim is to open a subset of plants to assign every customer taking into account their preferences. Customers rank the plants in strict order and are assigned to their favorite open plant, and the objective is to minimize the location plus allocation costs. Here, we study a generalization of the SPLPO named the Capacitated Facility Location Problem with Customer Preferences (CFLCP) where a limited number of customers can be allocated to each facility. We consider the global preference maximization setting, where the customers preferences are globally maximized. For this setting, we define three new types of stable allocations, namely customer stable, pairwise stable and cyclic-coalition stable allocations, and we provide two mixed-integer linear formulations for each setting. In particular, our cyclic-coalition stable formulations are Pareto optimal in a global-preference maximization setting, in the sense that no customer can improve their allocation without making another one worse off. We provide extensive computational experiments and compare the quality of our allocations with previous ones defined in the literature. As an additional result, we present a novel formulation that provides Pareto optimal matchings in the Capacitated House Allocation problem of maximum cardinality.

\end{abstract}
	
\begin{keyword}
Location With Preferences; Plant Location Problem; Capacitated Facility Location; Cyclic-Coalition Stable Allocations; Stable Matchings; Combinatorial Optimization; Matchings Under Preferences; Pareto Optimal Matchings
\end{keyword}

\end{frontmatter}

\section{Introduction} \label{sec:intro}
The Simple Plant Location Problem with Order (SPLPO) is a generalization of the well-known Facility Location Problem (FLP) that takes into account customers' preferences over the plants in the allocation process. The aim of SPLPO is to open a number of facilities from a set of candidate ones and allocate every customer, minimizing the overall cost of location plus allocation. Customers have certain preferences over the candidate locations that may be related to their size, distance, and other features, and they express their preferences ranking the candidate locations from the best to the worst (ties are not allowed in this setting). Once the plants are installed, each customer is assigned to their favorite open plant. It is assumed that the locator plays no role in the allocation of customers but has full knowledge of their rankings when making the location decision.

SPLPO was introduced in \cite{hanjoul1987}, where the state-of-the-art formulation available to this day is provided and a heuristic is developed to tackle small instances. \cite{hansen2004} introduce a bilevel formulation that is then reformulated into a linear model in several ways, the best being the same formulation given by Hanjoul and Peeters. \cite{canovas2007} build upon the work from \cite{hanjoul1987}, developing some valid inequalities as well as a basic preprocessing analysis, and \cite{vasil2009} also derive new valid inequalities based on the subyacent set packing polytope. The efficiency of said family of valid inequalities is checked in \cite{vasilyev2010} by implementing a cutting plane algorithm to find an optimal solution. A simulated annealing method is used to obtain the upper bounds in these exact methods. \cite{vasilyev2013} carry out a polyhedral analysis of the problem, stating which inequalities are facet defining and developing a new set of facet-defining valid inequalities that are tested through extensive computational experiments. Very recently, \cite{cabezas2023} study some advantages of using a Lagrangian and semi-Lagrangian approach for the SPLPO, and they use the solution of a Lagrangian relaxation algorithm as the starting point of a semi-Lagrangian relaxation method. Then they apply a variable fixing heuristic to find good feasible solutions (often the optimal solution) in a so-called semi-Lagrangian relaxation heuristic algorithm. Note that, if the ranking is given by the distance to the facilities, then the preference constraints included in \cite{hanjoul1987} and subsequent works match the well-known Closest Assignment Constraints (CAC). See \cite{espejo2012} for a detailed comparison of CAC in integer programming.

In this work, we study an extension of SPLPO where each facility has a capacity constraint, that is, a limited number of customers that can be allocated to it. The problem is known as Capacitated Facility Location with Customer's Preferences (CFLCP). This is a much more general setting with many more realistic applications, such as assigning children to schools, students to universities, and patients to hospitals, among others. Its counterpart without customer preferences, the Capacitated Facility Location (CFL) problem, has been thoroughly studied  \citep[without being exhaustive, see e.g.][]{avella2009, gortz2012, fernandez2015, fischetti2016}. However, the combination of a cardinality constraint and customers with preferences makes the allocation rule unclear. In fact, knowing the preferences of customers and taking them into account when selecting the facilities to be opened is not enough to properly define the assignments in the presence of capacity constraints. Thus, additional assumptions regarding the customers' preferences and willingness to exchange locations are necessary and result in various allocation rules that give rise to different problems and formulations. This might be the reason why there are only a handful of references in the location literature that combine these two ingredients. 

There are two main approaches to customer allocation. In the first, customers are assigned to their favorite open plant, as in  SPLPO. This approach corresponds to an \textit{individual preference maximization} setting. Naturally, CAC are to be used to model the problem and facilities that are overloaded cannot ever be opened (note that the overloading of a facility depends on the rest of the open facilities). To the best of our knowledge, there are only two works focusing on this approach. \cite{busing2022} analyze the CFLCP if it is defined on a graph, where the assignment costs are equal to the distances between the customer nodes and the facility nodes. They derive some settings that are NP-complete and some that are
polynomially solvable, closing the research gap regarding the computational complexity of FLPs with CAC mixed with capacity and revenue constraints. \cite{busing2025} build upon the previous work and contribute two novel preprocessing
methods, which reduce the size of the considered integer programming formulation. Furthermore, they introduce sets of valid inequalities that decrease the integrality gaps.

The second approach is the one followed in this article and consists in a \textit{global preference maximization} setting. In this case, any facility can be opened, and if a facility is overloaded customers can be served at their next favorite one. The few works that follow this approach all introduce bilevel formulations with a second-level global-preference maximization problem. The authors of \cite{calvete2020} introduce a bilevel formulation that is then reformulated into a single-level mixed-integer problem using duality theory. They also develop a matheuristic based on an evolutionary algorithm and compare the performance of both approaches. We build upon the work developed in this article, as it serves as motivation for the formulations hereby introduced. In \cite{casas2018}, a very similar bilevel model is proposed for a slightly different setting where each customer has a demand and the capacity of a facility refers to the amount of demand it can deal with. This fact increases the difficulty of the resultant second-level formulation, which is a generalized assignment problem and hence NP-hard. To overcome this issue, the authors define attainable bilevel solutions based on an efficient approximation of the inducible region and develop heuristic algorithms to tackle the problem. Finally, \cite{polino2024} develop a bilevel problem for extracurricular workshop planning that is a generalization of the previous where different types of facilities (in this case, workshops with different capacities) can be installed in the same location. They propose a single-level reformulation and two metaheuristics (a path-relinking metaheuristic and an iterated local search procedure).

%\red{As an example, \cite{espejo2009} define a \emph{total envy criterion} and develop several integer linear programming formulations that minimize the overall envy felt by demand points.}

\subsection{Our contribution}
We analyze how the inclusion of capacity constraints poses a challenge in the characterization of solutions. As stated, the maximization of individual customers' preferences may clash with the capacity restrictions and lead to unrealistic infeasible solutions depending on the application. Moreover, we raise awareness that a global preference maximization that relies on numeric preference values as introduced in the previous bilevel models is also disadvantageous, since it can lead to biased solutions where one customer is favored or disfavored over the rest.  So, the question arises: If numeric values must not characterize preferences in a global preference maximization objective function, how do we find a fair allocation that maximizes their preferences?
    
We give an answer to this question by replacing the second-level objective (maximizing customers' preferences) with allocations that meet different stability criteria. To this end, we introduce customer stable, pairwise stable, and cyclic-coalition stable allocations. These stable allocations meet fairness criteria in the sense that they provide unbiased solutions towards the customers. They are not based on numeric preference values, but rather on the ranking (strict ordering) of the plants given by the customers. As an additional advantage, characterizing these stable allocations can be made through constraints, so the models obtained are all mixed-integer linear or integer linear single-level models.

%The three stability criteria result in customer stable, pairwise stable and cyclic-coalition stable allocations. 
In a customer stable allocation, no customer can improve their position by switching to an undersubscribed plant. Therefore, customer stable solutions arise naturally in contexts where, once allocated, customers are unable to switch allocations unless their desired plant is undersubscribed. Pairwise stable settings are an extension of the previous that fit applications where customers are also allowed to swap locations, and they are willing to do so if both of them improve. Hence, pairwise stable allocations forbid assignments with pairs of clients in this situation. Lastly, cyclic-coalition stable solutions are an extension of pairwise stable solutions to any subset of customers. They are undominated in the sense that no customer can improve their allocation without making another one worse off. They constitute a non-biased global-preference-maximization setting.

These allocations are inspired by the literature on matching under preferences \citep[see][]{manlove2013} and extended here to account for the location of the facilities as well. The allocation problem in our setting (capacities and one-sided preferences) corresponds to the Capacitated House Allocation problem (CHA), where a series of \emph{applicants} are to be allocated to a series of \emph{houses} of a certain capacity, and the applicants have a preference list ranking a subset of the houses in strict order. Although the customer stable and pairwise stable are not defined for matchings, our cyclic-coalition stable allocation corresponds to a Pareto optimal matching in the CHA \citep[][Chapter 6]{manlove2013}. Note that there are no optimization models given in the matching literature or in the location literature to find any of these matchings. The fastest algorithms to find Pareto optimal matchings of maximum cardinality are given in \cite{sng2008} along with complexity results.

For each of the three settings, we provide an initial integer formulation and a reduced mixed-integer formulation with fewer binary variables that has some relevant properties. As an additional result, we derive a model to solve the maximum Pareto optimal CHA. We compare each initial model with its reduced counterpart by conducting extensive computational experiments with randomly created instances. Finally, we compare our stable solutions with the solutions provided in \cite{calvete2020} and show reductions of up to 17\% in the objective value. 

The paper is organized as follows. The notation and motivation for the problem are given in Section \ref{sec:motivation}. Three types of stable allocations are defined in Section \ref{sec:stableAllocations}. In Sections \ref{sec:customerstablemodels}-\ref{sec:cycliccoalitionstablemodels} we introduce initial and reduced customer, pairwise and cyclic-coalition stable formulations, respectively, and a maximum Pareto optimal formulation for CHA. Computational experiments are carried out in Section \ref{sec:compExp}, and some conclusions and future research are provided in Section \ref{sec:conclusion}.

\section{Notation and previous bilevel models for CFLCP} \label{sec:motivation}

In this section, we review two formulations proposed in the literature to tackle CFLCP and provide examples to illustrate the type of solutions they provide. Recall that the aim is to open a subset of facilities and to allocate all the customers satisfying the capacity constraints. Along with these requirements, we have two competing objectives which are the maximization of customers' preferences and the minimization of the location and allocation costs.

First, let us introduce the notation that will be used throughout the work. Let $\J$ be the set of potential plants or facilities to be opened, and $\I$ be the set of customers to be allocated. Each plant $j\in \J$ has an associated cost $f_j$ which refers to its opening, and for each pair $i,j$ of customer-facility there is an associated cost $g_{ij}$ of allocating customer $i$ to facility $j$. The capacity of each facility $j\in \J$ is $c_j$. W.l.o.g., it is assumed that $c_j < |I|$ for at least one facility $j$ (or else the problem becomes the (uncapacitated) SPLPO). Each customer $i\in \I$ has a ranking of the facilities from the best to the worst (ties are not allowed). The notation $j\prec_i j'$ states that $j$ is higher in the ranking than $j'$, that is, that $i$ \textit{prefers} $j$ to $j'$.%, then we denote  b product $i$ higher than product $j$, we say that $k$ \textit{prefers} $i$ to $j$, and we denote $i\prec_k j$. 

As is customary in the literature, we define binary variables $y_j$, $j\in \J$, equal to 1 if and only if the corresponding facility is open. Likewise, we define binary variables $x_{ij}$ for each customer $i\in \I$ and plant $j\in \J$, such that $x_{ij}=1$ if and only if customer $i$ is assigned to plant $j$. 

The model given in \citet{calvete2020} tackles the resolution of the CFLCP. To do so, the authors assume that each customer $i\in \I$ has a ranking of the facilities characterized by means of a set of predefined nonnegative preferences $p_{ij}\in \{1,\dots,|J|\}$, $j \in \J$, with $p_{ij} \neq p_{ij'}$ for $j\neq j'$. The lower the value, the greater the preference: $p_{ij} < p_{ij'} \Leftrightarrow j\prec_i j'$. Their bilevel model reads:
\begin{subequations} \label{BMCalvete}
\begin{align} 
\min_{\boldsymbol{y}} \quad & \sum_{j\in \J} f_jy_j + \sum_{i\in \I} \sum_{j\in \J} g_{ij}x_{ij}  \label{BMCalvete_OF1} \\
\text{s.t.} \quad &\sum_{j\in \J} c_jy_j \geq |I|, \label{BMCalvete_FeasibilityCut}\\
& y_j \in \{0,1\}, \quad \forall j\in\J, \label{BMCalvete_Ybinarias} \\
& x_{ij} \in \quad\arg\min_{\boldsymbol{x}}\quad  \sum_{i\in \I} \sum_{j\in \J} p_{ij}x_{ij}  \label{BMCalvete_OF2}  \\
    &\hspace{23mm}  \text{s.t.}\quad   \sum_{j\in \J} x_{ij} = 1, \quad \forall i\in \I, \label{BMCalvete_allocation} \\
&\hspace{23mm} \phantom{s.t.} \quad \sum_{i\in \I} x_{ij}  \le c_jy_j, \quad \forall j\in\J, \label{BMCalvete_capacity}\\
&\hspace{23mm}\phantom{s.t.} \quad x_{ij} \in \{0,1\}, \quad \forall i \in \I, j \in \J.  \label{BMCalvete_Xbinarias}
\end{align}
\end{subequations}

Model \eqref{BMCalvete} has a bilevel structure where the first level \eqref{BMCalvete_OF1}-\eqref{BMCalvete_Ybinarias} is the \emph{location} (i.e., the company's decision of the facilities to open) and the second level \eqref{BMCalvete_OF2}-\eqref{BMCalvete_Xbinarias} is the \emph{allocation} (that is, the assignment of the customers to the facilities). The objective function of the first level \eqref{BMCalvete_OF1} is to minimize the overall costs of opening plants and allocating customers. Constraints \eqref{BMCalvete_FeasibilityCut} state that there must be enough open plants to allocate every customer. This constraint is redundant, but it is added because it acts as a global feasibility cut guaranteeing that any feasible location $\boldsymbol{y}$ of the first-level problem admits a feasible allocation $\boldsymbol{x}$. The objective function of the second level \eqref{BMCalvete_OF2} is to maximize the preferences of customers globally. Constraints \eqref{BMCalvete_allocation} ensure that every customer is allocated to one plant, whereas the \textit{capacity constraints} \eqref{BMCalvete_capacity} guarantee that the capacity of each plant is not exceeded. Finally, constraints \eqref{BMCalvete_Ybinarias} and \eqref{BMCalvete_Xbinarias} express the binary character of the variables in the model. When constraints \eqref{BMCalvete_capacity} are included, the set of constraints ensuring that customers are allocated only to open plants,
\begin{equation} \label{BMCalvete_VI}
    \quad x_{ij} \le y_j, \quad \forall i \in \I, j \in \J,
\end{equation}
\noindent is redundant. However, they are known to act as valid inequalities, strengthening the linear relaxation bounds of the second-level problem. Model \eqref{BMCalvete} is reformulated in \cite{calvete2020}, obtaining a single-level mixed-integer linear model that can be solved using off-the-shelf solvers.

Note that the second-level problem \eqref{BMCalvete_OF2}-\eqref{BMCalvete_Xbinarias} is separable by customer in the uncapacitated SPLPO, since there are no linking constraints \eqref{BMCalvete_capacity} (but \eqref{BMCalvete_VI} should be included). On the other hand, if the capacity constraints \eqref{BMCalvete_capacity} are included in the first level, then again the second level problem is separable by customer, and thus each customer must be allocated to their most preferred open facility. As stated in \cite{calvete2020}, the main drawback of this approach is that, since customers choose \textit{individually}, the allocations given by the solutions of the second-level problem are very frequently infeasible due to the capacity restrictions. Hence, the proposal of the bilevel model \eqref{BMCalvete}, in which the preferences of the customers are \textit{globally} maximized and for any location satisfying \eqref{BMCalvete_FeasibilityCut} there exists a feasible allocation.

In the following example, we use several toy preference matrices to illustrate how the values of the parameters $p_{ij}$ can affect the allocations given by the second-level problem in model \eqref{BMCalvete} and provide biased solutions.

\begin{table}%[ht]
    \begin{subtable}[h]{0.45\textwidth}
    \centering
    \begin{tabular}{ccccc}
    \toprule    
     & Cust.\ 1 & Cust.\ 2 & Cust.\ 3 & Cust.\ 4  \\ \cmidrule(lr){2-5}
    Plant 1 & 1 &  1 & 1 & 1  \\ 
    Plant 2 & 2 &  3 & 4 & 10 \\ 
    Plant 3 & 3 &  9 & 16 & 100 \\
    Plant 4 & 4 &  27 & 64 & 1000 \\
    \bottomrule
\end{tabular}
       \caption{$p_{ij}\notin \{1,\dots,4\}$, $\forall i \in \I$, $j \in \J$}
       \label{tab:toyexample1a}
    \end{subtable}
\hfill
    \begin{subtable}[h]{0.45\textwidth}
    \centering
    \begin{tabular}{ccccc}
    \toprule    
     & Cust.\ 1 & Cust.\ 2 & Cust.\ 3 & Cust.\ 4  \\ \cmidrule(lr){2-5}
    Plant 1 & 1 &  1 & 2 & 3  \\ 
    Plant 2 & 2 &  2 & 4 & 2 \\ 
    Plant 3 & 3 &  4 & 3 & 1 \\
    Plant 4 & 4 &  3 & 1 & 4 \\
    \bottomrule
\end{tabular}
       \caption{$p_{ij}\in \{1,\dots,4\}$, $\forall i \in \I$, $j \in \J$}
       \label{tab:toyexample1b}
    \end{subtable}
\caption{Preference matrices for Example \ref{example1}. All the plants have capacity equal to 2.} 
\label{tab:toyexample1}
\end{table}

\begin{example} \label{example1}
Consider a toy instance with $|\I|=|\J|=4$ and capacities $c_j=2$ $\forall j\in \J$. Assuming that the costs of opening a plant greatly surpass those of allocating customers, i.e., $f_j \gg g_{ij}$, exactly two plants are open in an optimal solution of the bilevel problem. We are interested in the optimal allocations provided by the second-level problem \eqref{BMCalvete_OF2}-\eqref{BMCalvete_Xbinarias} for different solutions of the first-level problem.

First, consider the preference matrix $(p_{ij})$ given in Table \ref{tab:toyexample1a}. In this case, the ranking of the customers is clearly the same for all of them: $1 \prec_i 2 \prec_i 3 \prec_i 4$ $\forall i\in \I$. In theory, then, any feasible solution consists in allocating two customers to their favorite open plant, and the other two customers to their second favorite plant, so \textit{any} feasible allocation should be optimal for the second-level problem. However, due to the values $p_{ij}$ assigned to the preferences of each customer, only one allocation is optimal if two plants are open: customers 3 and 4 are allocated to their most preferred plant (between the two open ones), whereas customers 1 and 2 are allocated to their second favorite. This biased allocation where the preferences of customers 3 and 4 are \textit{always} fulfilled before the preferences of the other two customers regardless of the location of the plants occurs because the difference between the values $p_{ij}$ is greater for customers 3 and 4 than for customers 1 and 2.

In the previous setting it can be argued that the feasible allocations are influenced by the fact that the values $p_{ij}$ are very different among customers and they are not in the set $\{1,\dots,|\J|\}$. To illustrate how the allocations can be biased even if $p_{ij} \in \{1,2,\dots,|\J|\}$, consider the same setting with the preference matrix given in Table \ref{tab:toyexample1b}. 

Suppose plants 1 and 2 are open. Then customers 1, 2 and 3 would rather be allocated to plant 1 (and customer 4 to plant 2) but plant 1 has capacity equal to 2, so the preferences of one customer will not be maximized. Thus, in principle any of the three feasible allocations with customer 4 allocated to plant 2 is maximizing the preferences of three customers and should be optimal. In fact, two of these allocations are optimal: allocating customer 3 to plant 1 and customer 4 to plant 2 (and customers 1 and 2 each to a plant indistinctly) provides solutions with objective value 7.   However, if customers 1 and 2 are allocated to plant 1 and customers 3 and 4 to plant 2, the solution is not optimal for the second-level problem (and thus it is not feasible), since is has objective value equal to 8. In this case, customer 3 is favored over customers 1 and 2.

Consider now that plants 1 and 3 are open. Then a similar situation arises where customers 1, 2 and 3 prefer plant 1 over plant 3, but the \emph{only} optimal solution of the second-level problem consists in allocating customers 1 and 2 to plant 1 and customers 3 and 4 to plant 3. In this case, the preferences of customer 3 are dismissed in favor of those of customers 1 and 2.
\end{example}

As can be seen, although the second-level problem provides an optimal solution where customers maximize their preferences, it does not provide \emph{all} the solutions where this happens. The numbers $p_{ij}$ have no particular meaning: they are only included to express the ordering in the ranking of each customer. Yet, they influence the number and type of optimal solutions of the second level, limiting the number of feasible allocations for given first-level solutions in an undesirable manner. 
Finally, note that there is no bias when maximizing preferences expressed using numeric values in the uncapacitated SPLPO, because the preferences of the customers are satisfied individually, so the preferences of a customer are not compared to those of others.

\section{Three new types of stable allocations for the CFLCP} \label{sec:stableAllocations}
As we have hopefully conveyed through Example \ref{example1}, models that maximize numeric preferences are biased with respect to the preferences of certain customers, regardless of the actual numbers that we choose for parameters $p_{ij}$. Therefore, to formulate the problem we avoid the modeling of preferences using numeric values. Instead, we consider only the individual ranking of each customer, and we characterize solutions by the type of allocations they provide. To do so, we define three types of allocations: customer stable, pairwise stable and cyclic-coalition stable allocations. The concept of stability has been previously used in the context of matchings under preferences \citep[see][]{manlove2013}, and we redefine it here for our setting. We first introduce blocking customers, blocking pairs and blocking sets in an allocation. Recall that throughout the paper we assume that customers have strict preferences over plants.

\begin{defn} \label{def:blockingcustomer}
Given a feasible allocation $A$ in an instance of CFLCP, a customer $i$, allocated to plant $j$, is a \emph{blocking customer} if the following conditions are satisfied:
\begin{enumerate}
\item $i$ prefers plant $j'$ to his assigned plant $j$ in $A$, 

\item $j'$ is open, and

\item there are less than $c_{j'}$ customers allocated to $j'$ in $A$.
\end{enumerate}

Furthermore, $A$ is \emph{customer stable} if it has no blocking customers.
\end{defn}
In an instance of CHA, our customer stable allocation corresponds to a maximal trade-in-free matching.

\begin{defn}
Let $A$ be an allocation in an instance of CFLCP. A pair of customers $(i,i')$ allocated respectively to plants $j,j'$ ($j \neq j'$) is a \emph{blocking pair} for $A$ if $i$ prefers $j'$ over his current allocation $j$ and $i'$ prefers $j$ over $j'$. 

Furthermore, $A$ is said to be \emph{pairwise stable} if it does not admit blocking customers or blocking pairs.
\end{defn}
\begin{defn}
A set of customers $\bar{I}=\{i_1, \dots, i_n\}$, $|\bar{I}| \ge 3$, allocated to plants $\bar{J}=\{j_1, \dots, j_n\}$ respectively in an allocation $A$, is a \emph{blocking set} if there exists a bijection $\gamma: \bar{I} \rightarrow \bar{J}$ such that every customer $i_t$ prefers $\gamma(i_t)$ to its current allocation $j_t$.

Moreover, $A$ is \emph{cyclic-coalition stable} if it has no blocking customers, blocking pairs, or blocking sets.
\end{defn}

In an instance of CHA, our cyclic-coalition stable allocation corresponds to a maximal, trade-in-free and cyclic coalition-free matching and our pairwise stable allocation is not defined.

By extension, a formulation for the CFLCP is customer stable (resp.\ pairwise stable, cyclic-coalition stable) if any feasible allocation is customer stable (resp.\ pairwise stable, cyclic-coalition stable).

These three types of allocations fit different applications, depending on the circumstances under which we assume that customers can change their allocation. Intuitively, when we seek for a customer stable allocation we are assuming that, once allocated, customers can only change their plant for another one that they prefer over their current one if it is undersubscribed. As for pairwise stable and cyclic-coalition stable allocations, they fit applications where customers are also able to swap allocations (in the case of pairwise stable) or \emph{rotate} (for three or more customers in the cyclic-coalition stable allocations) if, in doing so, all of them improve in terms of their preferences (hence the name Pareto optimal in the related matching literature). Note that the bilevel model provided in \citep{calvete2020} provides only cyclic-coalition stable allocations, but not all of them (many of them are rendered infeasible by the bilevel model).

\begin{remark}
    Model \eqref{BMCalvete} provides cyclic-coalition stable allocations. In fact, consider an allocation $A$ that is not cyclic-coalition stable. If there is a blocking customer $i$, then $i$ is allocated to $j$ but prefers $j'$, and $j'$ is undersubscribed. Hence, $p_{ij'}<p_{ij}$ and reassigning $i$ to $j$ produces an allocation with strictly lower objective value \eqref{BMCalvete_OF2}. Alternatively, there exists a blocking pair or a blocking set of customers $\bar{I}=\{i_1, \dots, i_n\}$, $|\bar{I}| \ge 2$, allocated to different plants $\bar{J}=\{j_1, \dots, j_n\}$, and a bijection $\gamma$ such that $i_t$ prefers $\gamma(i_t)$ to its current allocation $j_t$. But this implies $p_{i_t\gamma(i_t)} < p_{i_tj_t}$ for all $t\in \{1,\dots,n\}$, so there exists a feasible allocation with strictly lower objective value \eqref{BMCalvete_OF2}. Hence, in any case the initial allocation $A$ is infeasible for \eqref{BMCalvete}.
\end{remark}
 \begin{remark}
    Not every cyclic-coalition stable allocation is feasible for model \eqref{BMCalvete}. To illustrate this, consider Example \ref{example1} for the preference matrix given in Table \ref{tab:toyexample1b}. Consider the allocation given by customers 1 and 2 allocated to plant 1, and customers 3 and 4 allocated to plant 2 (only plants 1 and 2 are open). This allocation is infeasible for model \eqref{BMCalvete}, but is cyclic-coalition stable.
\end{remark}
In the following sections, we introduce customer stable, pairwise stable and cyclic-coalition stable formulations for the CFLCP. 
%\cd{The interested reader may notice that we do not consider a third stable allocation (more restrictive than the customer stable allocation but less restrictive than the strong one) where only blocking customers and \emph{blocking pairs} (i.e.\ pairs of customers that would rather swap allocations) are forbidden. In fact, this problem is known (in the matching literature) as finding stable matchings for the Hospitals/Residents problem, and recently there have been proposed models to solve it (see \cite{delorme2019}). Esto no está bien aquí, hay varias diferencias: las listas de preferencias tb vienen de los hospitales, y un blocking pair es que hay un hospital que quiere más a un residente y un residente que quiere más a ese hospital, no son como aquí que dependen exclusivamente de los clientes. De todas formas, creo que lo más interesante es citar toda esta rama de la literatura en la intro, ya que son los que hablan más de estabilidad y de los que se pueden mirar modelos.}

\section{Customer stable models for the CFLCP} \label{sec:customerstablemodels}
In this section, we begin by introducing a formulation that minimizes the overall cost of locating plants and allocating customers, and that provides customer stable allocations. To do so, consider the notation and variables introduced in Section \ref{sec:motivation}. To replace the numeric preferences, for each customer $i\in \I$ we use the previously defined total ordering $\prec_i$ such that $j\prec_i j'$ (resp.\ $j \succ_i j'$) if customer $i$ prefers plant $j$ over plant $j'$ (resp.\ plant $j'$ over plant $j$). Furthermore, $j\preceq_i j'$ if customer $i$ prefers plant $j$ over plant $j'$ or $j=j'$. With this notation, we introduce an integer linear formulation:
\begin{subequations} \label{CFLCP_CS}
\begin{align}
\texttt{(CFLCP)}_{\texttt{CS}} \quad \min_{\boldsymbol{x},\boldsymbol{y}} \quad & \sum_{j\in \J} f_jy_j + \sum_{i\in \I} \sum_{j\in \J} g_{ij}x_{ij}  \label{CFLCP_CS_FO} \\
\text{s.t.} \quad &\sum_{j\in \J} x_{ij} = 1, \quad \forall i\in \I, \label{CFLCP_CS_allocation}\\
& \sum_{i\in \I} x_{ij}  \le c_jy_j, \quad \forall j\in\J, \label{CFLCP_CS_capacity}\\
&x_{ij} \le y_j, \quad \forall i \in \I, j \in \J, \label{CFLCP_CS_allocationOpenPlants} \\
&  c_jy_j \leq c_j\sum_{j'\in \J:\atop j'\preceq_i j} x_{ij'} + \sum_{i'\in \I\setminus \{i\}} x_{i'j}, \quad \forall i\in \I, j\in \J, \label{CFLCP_CS_preferences} \\
& x_{ij} \in \{0,1\}, \quad \forall i \in \I, j \in \J,  \label{CFLCP_CS_xbinarias} \\
& y_j \in \{0,1\}, \quad \forall j \in \J.  \label{CFLCP_CS_ybinarias}
\end{align}
\end{subequations}

The objective \eqref{CFLCP_CS_FO} is to minimize the location and allocation costs. Constraints \eqref{CFLCP_CS_allocation} and \eqref{CFLCP_CS_capacity} ensure that all customers are allocated and the capacity of the plants is fulfilled. Constraints \eqref{CFLCP_CS_allocationOpenPlants} are redundant but are known to strengthen the linear relaxation. Constraints \eqref{CFLCP_CS_preferences} guarantee that the allocation is customer stable: when a plant $j$ is open ($y_j=1$), either customer $i$ is allocated to a plant $j'$ that he likes the same of more than $j$ (so $\sum_{j'\in \J: j'\preceq_i j} c_jx_{ij'} = c_{j}$), or there must be $c_j$ customers allocated to $j$ (i.e.\ plant $j$ is full). Constraints \eqref{CFLCP_CS_xbinarias}-\eqref{CFLCP_CS_ybinarias} ensure the binary nature of the variables.

\begin{prop}
Every feasible allocation provided by model ${\normalfont\texttt{(CFLCP)}_{\texttt{CS}}}$ is customer stable.
\end{prop}
\begin{proof}
Assume that allocation A is not customer stable, and let $i$ be a blocking customer, and $j, j'\in \J$ under the conditions of Definition \ref{def:blockingcustomer}: $x_{ij}=1$, $j' \prec_i j$ and $\sum_{i' \in \I} x_{i'j'} < c_{j'}$. Then for the pair $i\in\I$, $j'\in \J$, constraint \eqref{CFLCP_CS_preferences} is not satisfied: the left-hand side of \eqref{CFLCP_CS_preferences} is equal to $c_{j'}$ due to 2., whereas $\sum_{j''\in \J: j''\preceq_i j'} c_{j'}x_{ij''} = 0$ due to 1.\ and $\sum_{i'\in \I\setminus \{i\}} x_{i'j'} < c_{j'}$ due to 3. Hence, if an allocation is feasible for \eqref{CFLCP_CS}, it is customer stable.
\end{proof}

In the uncapacitated version of this problem, it is well-known that the integrality constraints on the $x$-variables can be relaxed. However, as we see in the following, this is not the case for the previous model.%two models addressing customer stable and pairwise stable allocations.

\begin{table}%[ht]
\begin{center}
    \begin{subtable}[h]{\textwidth}
    \begin{center}
    \begin{tabular}{ccccccccc}
    \toprule    
     & Cust.\ 1 & Cust.\ 2 & Cust.\ 3 & Cust.\ 4  & Cust.\ 5 & Cust.\ 6 & Cust.\ 7 & Cust.\ 8  \\ \cmidrule(lr){2-9}
    Plant 1 & 3 &  3 & 3 & 1 & 3 &  1 & 3 & 1 \\ 
    Plant 2 & 3 &  3 & 3 & 2 & 2 &  3 & 1 & 2\\ 
    Plant 3 & 3 &  3 & 1 & 1 & 1 &  1 & 2 & 3\\
    Plant 4 & 3 &  3 & 1 & 3 & 3 &  2 & 1 & 1\\
    \bottomrule
\end{tabular}
       \caption{Cost matrix $(g_{ij})$, $\forall i \in \I$, $j \in \J$.}
       \label{tab:counterexamplecosts}
       \end{center}
    \end{subtable}
%\hfill

    \begin{subtable}[h]{\textwidth}
    \begin{center}
    \begin{tabular}{ccccccccc}
    \toprule    
     & Cust.\ 1 & Cust.\ 2 & Cust.\ 3 & Cust.\ 4   & Cust.\ 5 & Cust.\ 6 & Cust.\ 7 & Cust.\ 8  \\ \cmidrule(lr){2-9}
    Plant 1 & 3 &  3 & 2 & 1  & 4 &  3 & 2 & 4  \\ 
    Plant 2 & 4 &  2 & 3 & 2  & 1 &  1 & 3 & 1 \\ 
    Plant 3 & 2 &  4 & 4 & 3  & 3 &  2 & 1 & 3 \\
    Plant 4 & 1 &  1 & 1 & 4  & 2 &  4 & 4 & 2 \\
    \bottomrule
\end{tabular}
       \caption{Preference matrix $(p_{ij})$, $\forall i \in \I$, $j \in \J$. }
       \label{tab:counterexampleprefs}
       \end{center}
    \end{subtable}
\caption{Small instance of the CFLCP. $c_j=3$ and $f_j=2$ $\forall j \in \J$.} 
\label{tab:counterexample}
\end{center}
\end{table}

\begin{prop}
    The integrality constraints on the $x$-variables \eqref{CFLCP_CS_xbinarias} cannot be relaxed in model ${\normalfont\texttt{(CFLCP)}_{\texttt{CS}}}$.
\end{prop}
\begin{proof}
Consider an instance such that $|\I|=8$, $|\J|=4$, the capacity of all plants is equal to 3, and the cost $f_j$ of opening plant $j$ is equal to 2 for any $j\in \J$. The allocation cost matrix and the preference matrix can be observed in Table \ref{tab:counterexample}. The previous instance constitutes a counterexample. The optimal solution given by formulation $\texttt{(CFLCP)}_{\texttt{CS}}$ with constraints \eqref{CFLCP_CS_xbinarias} relaxed has objective value equal to 18.2. Since all coefficients in the objective function are integer, there are no feasible solutions with integer $x$ variables that produce this optimal value. Actually, the true optimal value of this instance is 19. 
\end{proof}

% \begin{prop}
%     The integrality constraints on the $x$-variables \eqref{CFLCP_CS_xbinarias} and \eqref{CFLCP_PAIRS_xbinarias} cannot be relaxed in models ${\normalfont\texttt{(CFLCP)}_{\texttt{CS}}}$ and ${\normalfont \texttt{(CFLCP)}_{\texttt{PW}}}$, respectively.
% \end{prop}
% \begin{proof}
%      Consider an instance such that $|\I|=8$, $|\J|=4$, the capacity of all plants is equal to 3, and the cost $f_j$ of opening the plant $j$ is equal to 2 for any $j\in \J$. The allocation cost matrix and the preference matrix can be observed in Table \ref{tab:counterexample}. The previous instance constitutes a counterexample for both settings. The optimal solution given by formulation $\texttt{(CFLCP)}_{\texttt{CS}}$ with constraints \eqref{CFLCP_CS_xbinarias} relaxed (resp.\ ${\rm \texttt{(CFLCP)}_{\texttt{PW}}}$ with \eqref{CFLCP_PAIRS_xbinarias} relaxed) is 18.2 (resp.\ 18.8). Since all coefficients in the objective function are integer in both models, clearly there are no feasible solutions with integer $x$ variables that produce these optimal values. Actually, the true optimal value of this instance (in both models) is 19. \red{REVISADO! SEPARAR EN DOS}
% \end{proof}

In the following, we present an alternative customer stable model for which the integrality constraints on the $x$-variables can be relaxed. To state it, we use an additional set of binary variables $u_j$, $j\in \J$, equal to 1 if and only if plant $j$ is full in the solution. 
\begin{subequations} \label{CFLCP_CSR}
\begin{align}
\texttt{(CFLCP)}_{\texttt{CS}}^\texttt{R} \quad \min_{\boldsymbol{x},\boldsymbol{y}} \quad & \sum_{j\in \J} f_jy_j + \sum_{i\in \I} \sum_{j\in \J} g_{ij}x_{ij}  \label{CFLCP_CSR_FO} \\
\text{s.t.} \quad &\sum_{j\in \J} x_{ij} = 1, \quad \forall i\in \I, \label{CFLCP_CSR_allocation}\\
& \sum_{i\in \I} x_{ij}  \le (c_j-1)y_j + u_j, \quad \forall j\in\J, \label{CFLCP_CSR_capacity1}\\
& c_ju_j \le \sum_{i\in \I} x_{ij}, \quad \forall j\in\J, \label{CFLCP_CSR_capacity0}\\
&x_{ij} \le y_j, \quad \forall i \in \I, j \in \J, \label{CFLCP_CSR_allocationOpenPlants} \\
&  y_j \le u_j + \sum_{j'\in \J:\atop j'\preceq_i j} x_{ij'}, \quad \forall i\in \I, j\in \J, \label{CFLCP_CSR_preferences} \\
& x_{ij} \in \{0,1\}, \quad \forall i \in \I, j \in \J,  \label{CFLCP_CSR_xbinarias} \\
& u_j, y_j \in \{0,1\}, \quad \forall j \in \J.  \label{CFLCP_CSR_ybinarias}
\end{align}
\end{subequations}

Constraints \eqref{CFLCP_CSR_capacity1} and \eqref{CFLCP_CSR_capacity0} are now the \emph{capacity constraints}: \eqref{CFLCP_CSR_capacity1} ensure that $u_j=1$ when plant $j$ is full, whereas \eqref{CFLCP_CSR_capacity0} guarantee that $u_j=0$ when $j$ is undersubscribed. In addition, constraints \eqref{CFLCP_CSR_preferences} prohibit blocking customers. When $y_j=1$, either the plant is full ($u_j=1$) or customer $i$ is allocated to a plant $j'$ he likes the same or more than $j$ ($\sum_{j'\in \J: j'\preceq_i j} x_{ij'}=1$). 

\begin{prop} \label{pro: 3}
    The integrality constraints on the $x$-variables \eqref{CFLCP_CSR_xbinarias} can be relaxed in the model ${\normalfont\texttt{(CFLCP)}_{\texttt{CS}}^\texttt{R}}$.  
\end{prop}
\begin{proof}
    For given integer values of variables $y$ and $u$, let us consider a partition of $\J$ into three sets:
    \begin{itemize}
        \item $\J_1 := \{j\in \J: y_j=u_j=0 \}$ is the set of unopened plants,
        \item $\J_2 := \{j\in \J: y_j=1, u_j=0 \}$ is the set of undersubscribed plants,
        \item $\J_3 := \{j\in \J: y_j=u_j=1 \}$ is the set of full plants.
    \end{itemize}

    Using this partition, we can rewrite the constraints in formulation ${\normalfont\texttt{(CFLCP)}_{\texttt{CS}}^\texttt{R}}$ eliminating $y$-variables and redundant constraints as follows:
    \begin{subequations} \label{CFLCP_pCSR}
\begin{align}
\min_{\boldsymbol{x}}\quad & \sum_{j\in \J_2\cap \J_3} f_j + \sum_{i\in \I} \sum_{j\in \J} g_{ij}x_{ij}  \label{CFLCP_pCSR_FO} \\
\text{s.t.} \quad &\sum_{j\in \J} x_{ij} = 1, \quad \forall i\in \I, \label{CFLCP_pCSR_allocation}\\
& \sum_{i\in \I} x_{ij}  \le 0, \quad \forall j\in\J_1, \label{CFLCP_pCSR_capacityJ1}\\
& \sum_{i\in \I} x_{ij}  \le (c_j-1), \quad \forall j\in\J_2, \label{CFLCP_pCSR_capacityJ2}\\
& \sum_{i\in \I} x_{ij}  = c_j, \quad \forall j\in\J_3, \label{CFLCP_pCSR_capacityJ3}\\
&  1 \le  \sum_{j'\in \J:\atop j'\preceq_i j} x_{ij'}, \quad \forall i\in \I, j\in \J_2, \label{CFLCP_pCSR_preferences} \\
& x_{ij} \in \{0,1\}, \quad \forall i \in \I, j \in \J.  \label{CFLCP_pCSR_xbinarias} 
\end{align}
\end{subequations}
Assuming w.l.o.g that $J_2\neq \emptyset$, let $j_i$ be the favorite plant of customer $i\in \I$. Then constraint $(i,j_i)$ dominates the rest of the constraints $(i,j)$ from set  \eqref{CFLCP_pCSR_preferences}. Hence, combining constraints \eqref{CFLCP_pCSR_allocation} and \eqref{CFLCP_pCSR_preferences}, we can further reduce the formulation to:
    \begin{subequations} \label{CFLCP_ppCSR}
\begin{align}
\min_{\boldsymbol{x}} \quad & \sum_{j\in \J_2\cap \J_3} f_j + \sum_{i\in \I} \sum_{j\in \J} g_{ij}x_{ij}  \label{CFLCP_ppCSR_FO} \\
\text{s.t.} \quad &\sum_{j\in \J:\atop j \preceq_i j_i} x_{ij} = 1, \quad \forall i\in \I, \label{CFLCP_ppCSR_allocation}\\
& \sum_{i\in \I} x_{ij}  \le 0, \quad \forall j\in\J_1, \label{CFLCP_ppCSR_capacityJ1}\\
& \sum_{i\in \I} x_{ij}  \le (c_j-1), \quad \forall j\in\J_2, \label{CFLCP_ppCSR_capacityJ2}\\
& \sum_{i\in \I} x_{ij}  = c_j, \quad \forall j\in\J_3, \label{CFLCP_ppCSR_capacityJ3}\\
& x_{ij} \in \{0,1\}, \quad \forall i \in \I, j \in \J.  \label{CFLCP_ppCSR_xbinarias} 
\end{align}
\end{subequations}
The constraints matrix in \eqref{CFLCP_ppCSR} is totally unimodular \cite[Corollary 2.8. for the partition of rows given by constraints $\{\eqref{CFLCP_ppCSR_allocation}$ and $\{\eqref{CFLCP_ppCSR_capacityJ1}-\eqref{CFLCP_ppCSR_capacityJ3}\}$]{wolsey1999integer} and the right-hand side vector is integer, so the integrality constraints \eqref{CFLCP_ppCSR_xbinarias} can be relaxed \cite[Proposition 2.3.]{wolsey1999integer}.
\end{proof}

Note that since usually $|\J| \ll |\I|$, the number of variables in model \eqref{CFLCP_CSR} is larger compared to model  \eqref{CFLCP_CS}, but the number of binary variables has been drastically reduced from $|\J| \cdot (|\I|+1)$ to only $2|\J|$.

\section{Pairwise stable models for the CFLCP} \label{sec:pairwisestablemodels}

We begin the section with a pairwise stable formulation that uses the initial sets of variables $y_j$, $\forall j\in \J$, for the location of the plants and $x_{ij}$, $\forall i\in \I$, $j\in\J$, for the allocation of the customers:
\begin{subequations} \label{CFLCP_PAIRS}
\begin{align}
\texttt{(CFLCP)}_{\texttt{PW}} \quad \min_{\boldsymbol{x},\boldsymbol{y}} \quad & \sum_{j\in \J} f_jy_j + \sum_{i\in \I} \sum_{j\in \J} g_{ij}x_{ij}  \label{CFLCP_PAIRS_FO} \\
\text{s.t.} \quad &\sum_{j\in \J} x_{ij} = 1, \quad \forall i\in \I, \label{CFLCP_PAIRS_allocation}\\
& \sum_{i\in \I} x_{ij}  \le c_jy_j, \quad \forall j\in\J, \label{CFLCP_PAIRS_capacity}\\
&x_{ij} \le y_j, \quad \forall i \in \I, j \in \J, \label{CFLCP_PAIRS_allocationOpenPlants} \\
&  c_j(y_j + x_{ij'} -1 ) \le \sum_{i'\in \I\setminus \{i\}: \atop j \prec_{i'} j'} x_{i'j}, \quad \forall i\in \I, j,j'\in \J: j \prec_i j', \label{CFLCP_PAIRS_bloquingpair1} \\
& x_{ij} \in \{0,1\}, \quad \forall i \in \I, j \in \J,  \label{CFLCP_PAIRS_xbinarias} \\
& y_j \in \{0,1\}, \quad \forall j \in \J.  \label{CFLCP_PAIRS_ybinarias}
\end{align}
\end{subequations}

Constraints \eqref{CFLCP_PAIRS_allocation}-\eqref{CFLCP_PAIRS_allocationOpenPlants} are the same as in formulation $\texttt{(CFLCP)}_{\texttt{CS}}$. In addition, constraints \eqref{CFLCP_PAIRS_bloquingpair1} guarantee that there are no blocking customers or blocking pairs. If $y_j + x_{ij'} -1=1$, then $j$ is open, $i$ is allocated to $j'$ and $i$ prefers $j$ over $j'$. Then there must exist $c_j$ customers allocated to $j$ that also prefer $j$ over $j'$.

Constraints \eqref{CFLCP_PAIRS_bloquingpair1} in $\texttt{(CFLCP)}_{\texttt{PW}}$ share some similarities with those in the integer programming model presented in \cite{kwanashie2014} and strengthened in \cite{delorme2019} for the Hospitals/Residents problem. However, their definition of blocking pair is different from the one stated here: we consider a pair of customers (which corresponds with a pair of residents), whilst they consider a pair hospital-resident to define the blocking pair because in their setting both the hospitals and the residents have preferences. Hence, the constraints ensuring a pairwise stable allocation have a different meaning in both settings.

\begin{prop}
    The integrality constraints on the $x$-variables \eqref{CFLCP_PAIRS_xbinarias} cannot be relaxed in model ${\normalfont \texttt{(CFLCP)}_{\texttt{PW}}}$.
\end{prop}
\begin{proof}
     As a counterexample, consider the instance given in Table \ref{tab:counterexample} with $|\I|=8$, $|\J|=4$, the capacity of all plants is equal to 3, and the cost $f_j$ of opening plant $j$ is equal to 2 for any $j\in \J$. The optimal solution given by formulation ${\rm \texttt{(CFLCP)}_{\texttt{PW}}}$ with \eqref{CFLCP_PAIRS_xbinarias} relaxed has objective value equal to 18.8, whereas the  optimal value of this instance is 19. 
\end{proof}
To provide a reduced pairwise stable formulation, we use variables $u_j$, $\forall j\in \J$, defined in the previous section for model \eqref{CFLCP_CSR}. We also define an additional set of binary variables that help avoid blocking pairs. Thus, $v_{jj'}\in\{0,1\}$ $\forall j,j'\in \J$, $j\neq j'$, such that $v_{jj'}=1$ if there exists a customer $i\in \I$ who prefers $j$ over $j'$ but is assigned to $j'$.
\begin{subequations} \label{CFLCP_PAIRSR}
\begin{align}
\texttt{(CFLCP)}_{\texttt{PW}}^\texttt{R} \quad \min_{\boldsymbol{x},\boldsymbol{y},\boldsymbol{v}} \quad & \sum_{j\in \J} f_jy_j + \sum_{i\in \I} \sum_{j\in \J} g_{ij}x_{ij}  \label{CFLCP_PAIRSR_FO} \\
\text{s.t.} \quad &\sum_{j\in \J} x_{ij} = 1, \quad \forall i\in \I, \label{CFLCP_PAIRSR_allocation}\\
& \sum_{i\in \I} x_{ij}  \le (c_j-1)y_j+u_j, \quad \forall j\in\J, \label{CFLCP_PAIRSR_capacity1}\\
& c_ju_j \le \sum_{i\in \I} x_{ij}, \quad \forall j\in\J, \label{CFLCP_PAIRSR_capacity0}\\
&x_{ij} \le y_j, \quad \forall i \in \I, j \in \J, \label{CFLCP_PAIRSR_allocationOpenPlants} \\
&  y_j \le u_j + \sum_{j'\in \J:\atop j'\preceq_i j} x_{ij'}, \quad \forall i\in \I, j\in \J, \label{CFLCP_PAIRSR_blockingcust} \\
&  v_{jj'} + v_{j'j} \leq 1, \quad \forall j\neq j'\in \J, \label{CFLCP_PAIRSR_blockingpair1} \\
&  x_{ij'} \le v_{jj'}, \quad \forall i\in \I, j,j'\in \J: j \prec_i j', \label{CFLCP_PAIRSR_blockingpair2} \\
& x_{ij} \in \{0,1\}, \quad \forall i \in \I, j \in \J,  \label{CFLCP_PAIRSR_xbinarias} \\
& u_j, y_j \in \{0,1\}, \quad \forall j \in \J,  \label{CFLCP_PAIRSR_ybinarias} \\
& v_{jj'} \in \{0,1\}, \quad \forall j\neq j' \in \J.  \label{CFLCP_PAIRSR_vbinarias} 
\end{align}
\end{subequations}
Constraints \eqref{CFLCP_PAIRSR_allocation}-\eqref{CFLCP_PAIRSR_blockingcust} are the same as in model ${\normalfont\texttt{(CFLCP)}_{\texttt{CS}}^\texttt{R}}$, and provide a formulation that yields customer stable allocations. Moreover, constraints \eqref{CFLCP_PAIRSR_blockingpair2} ensure that if there is a customer $i$ assigned to $j'$ but prefers $j$ over $j'$, then $v_{jj'}=1$. Together with \eqref{CFLCP_PAIRSR_blockingpair1}, these constraints forbid blocking pairs.

\begin{prop}
    The integrality constraints on the $x$-variables \eqref{CFLCP_PAIRSR_xbinarias} can be relaxed in model ${\normalfont \texttt{(CFLCP)}_{\texttt{PW}}^\texttt{R}}$.
\end{prop}
\begin{proof}
For given integer values of variables $y$, $u$ and $v$, observe that the constraints in model $\texttt{(CFLCP)}_{\texttt{PW}}^\texttt{R}$ involving the variables $x$ match those in $\texttt{(CFLCP)}_{\texttt{PW}}$, except for constraints \eqref{CFLCP_PAIRSR_blockingpair2}. 

From the proof of Proposition \ref{pro: 3}, it follows that the constraint matrix, excluding \eqref{CFLCP_PAIRSR_blockingpair2}, is totally unimodular. The addition of constraints \eqref{CFLCP_PAIRSR_blockingpair2} does not alter this property, as it merely appends rows of zeros with a single one into the matrix, preserving total unimodularity.

The right-hand side vector is integer, so the integrality constraint \eqref{CFLCP_PAIRSR_xbinarias} can be relaxed \cite[Proposition 2.3]{wolsey1999integer}.
\end{proof}

We compare formulations $\texttt{(CFLCP)}_{\texttt{CS}}$ and $\texttt{(CFLCP)}_{\texttt{PW}}$ with their counterparts $\texttt{(CFLCP)}_{\texttt{CS}}^\texttt{R}$ and $\texttt{(CFLCP)}_{\texttt{PW}}^\texttt{R}$ in Section \ref{sec:compExp}.

\section{Cyclic-coalition stable models for the CFLCP} \label{sec:cycliccoalitionstablemodels}
In this section, we minimize the overall cost of location and allocation considering only cyclic-coalition stable allocations for customers.

\subsection{Basic Model}
 Our first integer linear model is based on the following characterizations of cyclic-coalition stable allocations:

\begin{lemma} \label{lem:permutation}
 Consider an instance of {\rm CFLCP} with a set of customers $\I=\{1,2,\dots,|\I|\}$ and a set of plants $\J$. Let $\J' \subseteq \J$ with $|\J'| \ge |\I|$ and let $\sigma: I \rightarrow I$ be a permutation of the customers in $I$. Next, consider the allocation that results from allocating customers in the order given by $\sigma$ (that is, first $\sigma(1)$, then $\sigma(2)$ and so on) to their favorite plant (not full) in $\J'$. The resulting allocation is cyclic-coalition stable.   
\end{lemma}
\begin{proof}
 Let $\bar{I}\subseteq \I$ be a set of customers allocated to different plants $\bar{J}$, respectively, following the order given by permutation $\sigma$. Let $\bar{i}:= \min  \{i\in \bar{I}\}$. Then customer $\sigma(\bar{i})$ was allocated to his favorite plant among the ones in $\bar{J}$ (none of them was at full capacity because the rest of the customers in $\bar{I}$ were allocated afterwards). Hence, $\bar{i}$ prefers their choice to that of the customers in $\bar{I}$, so $\bar{I}$ is not a blocking set (or a blocking pair if $|\I|=2$). Since each customer is allocated to their favorite undersubscribed plant, the fact that there are no blocking customers is straightforward.
\end{proof}

\begin{lemma}
Every blocking set contains a blocking subset in which the assigned plants are all distinct.
\end{lemma}

\begin{proof}
Let \( \bar{\I} = \{i_1, i_2, \dots, i_s\} \subseteq \I \), with corresponding assigned plants \( j_1, j_2, \dots, j_s \in \J \), and let \( \gamma : \{1, 2, \dots, s\} \to \{1, 2, \dots, s\} \) be a bijection such that each customer \( i_t \) prefers plant \( j_{\gamma(t)} \) over its current allocation \( j_t \), as in the definition of blocking set. Suppose, by contradiction, that there exist \( a, b \in \{1, \dots, s\} \), with \( a \ne b \), such that \( j_a = j_b \).

If \( \gamma^r(a) \ne b \) for all \( r = 1, 2, \dots, s-1 \), where \( \gamma^r \) denotes the composition of \( \gamma \) with itself \( r \) times, consider the subset of customers
\[
    \{i_{\gamma^r(a)} : r = 1, 2, \dots, s-1\}
\]
along with their assigned plants and the induced permutation. This subset excludes the specific repetition between \( j_a \) and \( j_b \).

If instead \( \gamma^r(a) = b \) for some \( r = 1, 2, \dots, s-1 \), let \( r^\star \) be the smallest such index. Consider the subset
\[
    \{i_{\gamma^r(a)} : r = 1, 2, \dots, r^\star\} \setminus \{i_a\}
\]
with their assigned plants and the induced permutation \( \gamma^\star \), modified so that \( \gamma^\star(b) = \gamma(a) \). This again eliminates the repetition between \( j_a \) and \( j_b \).

In both cases, the resulting subset remains a blocking set, since the induced permutation preserves the preference relations among the remaining customers. In particular, the subset has cardinality at least two.

Since the customer set is finite, we can repeat this procedure to eliminate all repeated plant assignments. Thus, we eventually obtain a blocking subset in which the assigned plants are all distinct or a blocking pair.
\end{proof}
 
The proposed model creates a permutation in the customers' set to allocate the customers, preventing any blocking customer/pair/set. With this aim, we use a set of binary variables $y_j$, $j\in \J$, equal to 1 if $j$ is open. Next, let $x_{ij}^r$, for $i,r\in \I$, $j\in \J$, be a binary variable equal to 1 if and only if the customer $i$ is assigned to the plant $j$ and $\sigma(i)=r$ in the permutation $\sigma$ that gives rise to the feasible allocation. Furthermore, let $u^r_j$, for $j\in \J$, $r\in \I$, be a binary variable that takes value 1 if and only if the plant $j$ fills up when the customer $i=\sigma^{-1}(r)$ is assigned. With these variables, the model is:
\begin{subequations} \label{CFLCProunds}
\begin{align} 
\texttt{(CFLCP)}_{\texttt{CC}} \quad \min_{\boldsymbol{x}, \boldsymbol{y}, \boldsymbol{u}} \quad & \sum_{j\in \J} f_jy_j + \sum_{i\in \I} \sum_{j\in \J} \sum_{r\in \I} g_{ij}x_{ij}^r  \label{CFLCPBrounds_OF1} \\
\text{s.t.} \quad & \sum_{j\in \J}\sum_{r\in\I} x_{ij}^r = 1, \quad \forall i\in \I, \label{CFLCPBrounds_allocation} \\
& \sum_{i \in \I}\sum_{j\in\J} x_{ij}^r = 1, \quad \forall r\in \I, \label{CFLCPBrounds_allocation2} \\
& \sum_{r\in \I} x_{ij}^r  \le y_j, \quad \forall i\in \I,j\in\J, \label{CFLCPBrounds_assignmenttoopenplants}\\
& \sum_{r\in \I} u_j^r \le y_j, \quad \forall j\in\J, \label{CFLCPBrounds_seagotasiseabre} \\
& \sum_{j\in \J} u_j^r \le 1, \quad \forall r\in\I, \label{CFLCPbrounds_seagotaunacadavez} \\
&  \sum_{i\in \I} \sum_{r'=1}^r x^{r'}_{ij} \le (c_j-1)y_j + \sum_{r'=1}^r u_j^{r'}, \quad \forall j\in\J, r \in \I  \label{CFLCPBrounds_uiguala0} \\
& c_j\sum_{r'=1}^r u_j^{r'} \le \sum_{i\in \I}\sum_{r'=1}^r x^{r'}_{ij}, \quad \forall j\in \J,r\in\I  \label{CFLCPBrounds_uiguala1} \\
&  y_j - \sum_{r'=1}^{r-1} u_j^{r'} \le \sum_{j' \preceq_i j}\sum_{r'=1}^r x_{ij'}^{r'} + \sum_{j'\in \J} \sum_{r'=r+1}^{|\I|}  x_{ij'}^{r'}, \quad \forall i,r\in \I,j\in\J, \label{CFLCPBrounds_preferences} \\
& y_j, x_{ij}^r, u_j^r \in \{0,1\}, \quad \forall i,r \in \I, j \in \J.  \label{CFLCPBrounds_binarias}
\end{align}
\end{subequations}
Constraints \eqref{CFLCPBrounds_allocation} ensure that $i$ is assigned to a single $j\in \J$, and together with  \eqref{CFLCPBrounds_allocation2} they guarantee that $\sigma$ is a permutation. Constraints \eqref{CFLCPBrounds_assignmenttoopenplants} guarantee that $i$ is allocated to an open plant. They are redundant when \eqref{CFLCPBrounds_uiguala0} are included, but are known to strengthen the linear relaxation in a similar manner to \eqref{CFLCP_CS_allocationOpenPlants}. Constraints \eqref{CFLCPBrounds_seagotasiseabre} impose that a plant can only fill up if it is open in the feasible solution. Constraints \eqref{CFLCPbrounds_seagotaunacadavez} ensure that up to one plant $j$ fills up with customer $\sigma(r)$. These constraints again are redundant when \eqref{CFLCPBrounds_uiguala0} and \eqref{CFLCPBrounds_uiguala1} are included, but significantly strengthen the formulation. Due to the capacity of the plants, it must hold $\sum_{i\in \I} \sum_{r'=1}^r x^{r'}_{ij} \le c_j$ for any $j$, $r$. When $\sum_{i\in \I} \sum_{r'=1}^r x^{r'}_{ij} = c_j$, \eqref{CFLCPBrounds_uiguala0} force $u^r_j=1$; and when $\sum_{i\in \I} \sum_{r'=1}^r x^{r'}_{ij} < c_j$, \eqref{CFLCPBrounds_uiguala1} guarantee that $u^r_j=0$. Constraints \eqref{CFLCPBrounds_preferences} ensure that the solution is cyclic-coalition stable, i.e.\ that every customer is allocated to their most preferred plant (among the ones that are open and not full) following the order given by permutation $\sigma$. 
%Consider a customer $i$ with $i=\sigma(r)$, that is, $i$ is the $r$-th customer to be allocated. Now consider $j\in \J$ and the corresponding constraint from \eqref{CFLCPBrounds_preferences}. If $j$ is not open or is full with the first $r-1$ customers already allocated, then the left-hand side of \eqref{CFLCPBrounds_preferences} is non-positive and the constraint holds. On the other hand, when $j$ is open and available, then the left-hand side of \eqref{CFLCPBrounds_preferences} is equal to 1, so one of the two sums of the right-hand side must be equal to one.
Consider a plant $j$ that is open ($y_j=1$) and not full when the first $r-1$ customers are allocated ($\sum_{r'=1}^{r-1} u_j^{r'}=0$). Then for any customer $i$ and for $r$, the left-hand side of the corresponding constraint in \eqref{CFLCPBrounds_preferences} is equal to 1, so one of the two sums of the right-hand side must be equal to one. In the first case,  $i=\sigma(\bar{r})$ with $\bar{r}\le r$ and $i$ is allocated either to plant $j$ or to a plant better than $j$ to $i$. In the second case, $\bar{r} > r$ and thus the constraint trivially holds.

\subsection{Reduced Model}
In this section, we propose a reduced model that provides the same set of cyclic-coalition stable allocations as model $\texttt{(CFLCP)}_{\texttt{CC}}$. Before introducing it, we illustrate the disadvantages of the previous one and the reasoning leading to the reduced model by means of an example.
%En esta sección, proponemos un modelo reducido (en términos de variables y restricciones) que proporciona las mismas soluciones (location-allocation) que el anterior. Antes de introducir el modelo, introducimos un ejemplo para explicar las desventajas del anterior y el razonamiento que nos conduce al modelo reducido.

\begin{table}%[ht]
    \centering
    \begin{tabular}{ccccccccc}
    \toprule    
     & Cust.\ 1 & Cust.\ 2 & Cust.\ 3 & Cust.\ 4  & Cust.\ 5 & Cust.\ 6 & Cust.\ 7 & Cust.\ 8\\ \cmidrule(lr){2-9}
    Plant 1 & 1 &  2 & 3 & 3 & 2 &  1 & 1 & 1  \\ 
    Plant 2 & 2 &  1 & 2 & 1 & 1 &  3 & 2 & 3 \\ 
    Plant 3 & 3 &  3 & 1 & 2 & 3 &  2 & 3 & 2 \\
    \bottomrule
\end{tabular}
\caption{Preference matrix for Example \ref{example2}. All the plants have capacity equal to 3.} 
\label{tab:toyexample2}
\end{table}

\begin{example} \label{example2} 
    Consider the toy instance with $\I=8$, $\J=3$, $c_j=3$ $\forall j\in \J$ and the preference matrix given in Table \ref{tab:toyexample2}. Since all the customers need to be allocated, all the facilities must be open in any feasible allocation and two of them fill up. 

    We can characterize every cyclic-coalition stable allocation following Lemma \ref{lem:permutation}. For instance, permutation $P=(1,2,3,4,6,7,5,8)$ results in customers $\{1,6,7\}$ being allocated to plant $1$, $\{2,4,5\}$ to plant $2$, and $\{3,8\}$ to plant $3$. The number of permutations of the customers in this toy instance amounts to $8!=40320$. However, many different permutations result in the same allocation. For instance, permutation $P'=(5, 2, 6, 7, 1, 3, 4, 8)$ results in the same allocation than $P$, as well as any permutation where customer $8$ gets to choose after customers $1$, $6$ and $7$. This is due to the fact that there are four customers whose first choice is plant $1$, so the plant is full when $8$ gets to choose. In fact, there are only 10 cyclic-coalition stable allocations for this instance, and all of them can be found in Table \ref{tab:toyexample2solution} (the allocation given by $P$ is the first one). Given that studying every single permutation does not seem very efficient, let us see how the amount of permutations to consider can be reduced. 

    The first six customers of $P$ are allocated to their favorite plant because all of the plants are available. However, plant $1$ fills up with customer $7$, so $5$ and $8$ can only choose between plants $2$ and $3$. Customer $5$ prefers plant $2$ over $3$ (casually, plant 2 is his favorite) and after his allocation plant $2$ is also full. Finally, $8$ is allocated to plant $3$. Clearly, any rearrangement of the first seven customers results in the same allocation, because their favorite plant is not full when they get to choose. So, as a representative of the permutation $P$, we can consider $\bar{P}=(1,6,7,2,4,5,3,8)$. In $\bar{P}$, the first three customers fill up plant $1$, the next three fill up plant $2$, and the last two are allocated to $3$. Moreover, the order in which customers $1$, $6$ and $7$ are allocated to plant $1$ is also irrelevant, so they could be allocated at a time (in one single \emph{round} or \emph{en bloc}). So, to obtain a cyclic-coalition stable allocation, it suffices to define a permutation giving the order in which the \emph{plants} are filled up (in the sense that the customers in any block are all allocated to the same plant), followed by the actual allocation of customers. An example of this ``permutation plus allocation'' is $\hat{P}=(1,2,3)$, where the order is established only among plants (plant 1 fills up first, then plant 2, and finally plant 3), together with the allocations $(\{1,6,7\},\{2,4,5\},\{3,8\})$. Note that $\hat{P}$ only results in a cyclic-coalition stable allocation provided that the first block of customers is allocated to their favorite plant out of the three available (which is $1$), and the next block to their favorite plant in the set $\{2,3\}$ ($2$ in this case).

    To further develop this idea, let us consider three different permutations of plants and the cyclic-coalition stable allocations that can be obtained. For $P_1=(1,2,3)$, there are $4$ customers $\{1,6,7,8\}$ whose favorite plant is $1$. If we allocate $\{1,6,7\}$ to plant $1$, then the remaining customers are $\{2,4,5\}$ (who prefer plant $2$ over plant $3$) and $\{3,8\}$ (who prefer plant $3$ over plant $2$), so allocating $\{1,6,7\}$ to plant $1$ results in one cyclic-coalition stable allocation (the first one in Table \ref{tab:toyexample2solution}, $\hat{P}$). If we still follow $P_1$ but allocate $\{1,6,8\}$ to plant $1$, then there are $4$ customers $\{2,4,5,7\}$ whose favorite plant is $2$, so we obtain the cyclic-coalition stable allocations $2-5$ in Table \ref{tab:toyexample2solution}. For a different plant permutation $P_2=(2,1,3)$, there are exactly three customers $\{2,4,5\}$ whose favorite plant is $2$. Out of the remaining five, $\{1,6,7,8\}$ prefer plant $1$, so in this case we obtain allocations $1$, $2$, $6$ and $10$ in Table \ref{tab:toyexample2solution} (note that permutations $P_1$ and $P_2$ might lead to the same allocation). Lastly, there is no cyclic-coalition stable allocation following the order $P_3=(3,1,2)$. This occurs because there is no set of $3$ customers whose favorite plant is $3$. 
    
%soluciones en las que se asignan dos clientes sin llenar plantas son intercambiables, se pueden ``resumir'' en una sola etapa. Es muy importante poder demostrar esto). Entonces, bastaría asignar a los clientes en \textit{bloques} o \textit{rondas}, una por cada planta que se llena, definiendo el orden en el que se llenan. Esta idea viene desarrollada a continuación.

To finish, let us study the allocation where customers $\{1,5,6\}$ fill up plant $1$, $\{2,4,7\}$ are allocated to plant $2$ and $\{3,8\}$ to plant $3$. This allocation is customer stable because none of the customers of plants $1$ or $2$ would rather be allocated to plant $3$ (which is the only plant with availability). However, it is not cyclic-coalition stable, since customers $5$ and $7$ would rather swap allocations. There is no permutation of customers giving rise to this allocation: if $5$ chooses before $7$, then when $5$ is allocated, neither plant $1$ nor $2$ are full, so $5$ is allocated to plant $1$ (and vice versa). Likewise, there is no permutation of plants giving rise to this allocation.%: if, say, plant $1$ fills up before plant $2$, then customer $5$ can choose between $1$ or $2$ so $5$ is allocated to $2$.

\end{example}

\begin{table}%[ht]
    \centering
    \begin{tabular}{c|ccc}
    \toprule    
     & $j$=1 & $j$=2 & $j$=3 \\ \cmidrule(lr){2-4}
    \multirow{10}{.8cm}{$i$} & 167 &  245 & 38  \\ 
      & 168 &  245 & 37 \\ 
      & 168 &  247 & 35  \\ 
      & 168 &  257 & 34 \\ 
      & 168 &  457 & 23 \\ 
      & 178 &  245 & 36  \\ 
      & 678 &  124 & 35 \\ 
      & 678 &  125 & 34  \\ 
      & 678 &  145 & 23 \\ 
      & 678 &  245 & 13 \\ 
    \bottomrule
\end{tabular}
\caption{Cyclic-coalition stable allocations for the instance depicted in Table \ref{tab:toyexample2}.} 
\label{tab:toyexample2solution}
\end{table}

%\cd{NUEVA IDEA: Como esta formulación no da los resultados que queremos, hay que hacer otra en la que se ``ordene'' a los clientes en la asignación a las plantas. Así se garantiza que, para una ordenación, todos los clientes se asignan (en ese orden) a su ``planta favorita no agotada'', y por tanto la solución es cyclic-coalition stable. Para ello, incluimos en las variables anteriores índices $r$ que indican el orden, y también variables nuevas para cuando se llena una   planta. Para reducir MUCHÍSIMO el tamaño de la formulación, basta notar (y probar) que solo hace falta definir las $r$ para los PLANTAS, ya que en realidad lo que hay que definir es ``el orden en el que se llenan las plantas'', puesto que esto define cómo van cambiando las preferencias (es decir, soluciones en las que se asignan dos clientes sin llenar plantas son intercambiables, se pueden ``resumir'' en una sola etapa. Es muy importante poder demostrar esto). Entonces, bastaría asignar a los clientes en \textit{bloques} o \textit{rondas}, una por cada planta que se llena, definiendo el orden en el que se llenan. Esta idea viene desarrollada a continuación. }

For the reduced formulation, we keep the customary sets of binary variables $y_j$, $j\in \J$ and $x_{ij}$, $i\in \I$, $j\in \J$ that determine the location and allocation. We also define a set of binary variables $u_j^r$, $j,r,\in \J$, equal to one if and only if plant $j$ is the $r$-th plant to fill up. Using these variables, the formulation is:
\begin{subequations} \label{CFLCProunds}
\begin{align} 
\texttt{(CFLCP)}_{\texttt{CC}}^{\texttt{R}} \quad \min_{\boldsymbol{y}, \boldsymbol{x}, \boldsymbol{u}} \quad & \sum_{j\in \J} f_jy_j + \sum_{i\in \I} \sum_{j\in \J} g_{ij}x_{ij}  \label{CFLCProunds_OF1} \\
\text{s.t.} \quad & \sum_{j\in \J} x_{ij} = 1, \quad \forall i\in \I, \label{CFLCProunds_allocation} \\
& x_{ij}  \le y_j, \quad \forall i\in \I,j\in\J, \label{CFLCProunds_assignmenttoopenplants}\\
& \sum_{r\in \J} u_j^r \le y_j, \quad \forall j\in\J, \label{CFLCProunds_seagotasiseabre} \\
& \sum_{j\in \J} u_j^1 \le 1, \label{CFLCProunds_seagotaunacadavez1} \\
& \sum_{j\in \J} u_j^r \le \sum_{j\in \J} u^{r-1}_j, \quad \forall r\in\J, r >1  \label{CFLCProunds_seagotaunacadavezresto} \\
&  \sum_{i\in \I} x_{ij} \le (c_j-1)y_j + \sum_{r\in \J} u_j^r, \quad \forall j\in\J  \label{CFLCProunds_uiguala0} \\
& c_j\sum_{r\in \J} u_j^r \le \sum_{i\in \I} x_{ij}, \quad \forall j\in\J  \label{CFLCProunds_uiguala1} \\
&  x_{ij} + \sum_{r'=1}^{r-1} u_j^{r'} + \sum_{j'\in \J:\atop j' \prec_i j} u_{j'}^r \le 1 + y_j, \quad \forall i\in \I,j,r\in\J, \label{CFLCProunds_preferencesFull} \\
&  y_j - \sum_{r\in\J} u_j^r \le \sum_{j' \preceq_i j} x_{ij'}, \quad \forall i\in \I,j\in\J, \label{CFLCProunds_preferencesEmpty} \\
& x_{ij} \in \{0,1\}, \quad \forall i \in \I, j \in \J,  \label{CFLCProunds_xbinarias} \\
& y_j, u_j^r \in \{0,1\}, \quad \forall j,r \in \J.  \label{CFLCProunds_binarias}
\end{align}
\end{subequations}

Constraints \eqref{CFLCProunds_allocation} and \eqref{CFLCProunds_assignmenttoopenplants} guarantee that each customer is assigned to an open facility. Constraints \eqref{CFLCProunds_seagotasiseabre} impose that a facility can only be filled if it is open, while \eqref{CFLCProunds_seagotaunacadavez1} and \eqref{CFLCProunds_seagotaunacadavezresto} enforce that facilities are filled sequentially, with only one being filled per round. Constraints \eqref{CFLCProunds_uiguala0} and \eqref{CFLCProunds_uiguala1} establish the capacity limitations. Moreover, \eqref{CFLCProunds_preferencesFull} prevent a facility assigned to a customer from being filled before another that the customer prefers, guaranteeing that there are no blocking pairs or blocking sets of customers. And \eqref{CFLCProunds_preferencesEmpty} require that if a facility is open and not full, then each customer must be served by a facility that is at least as preferred as that one, preventing blocking customers. Finally, constraints \eqref{CFLCProunds_xbinarias} and \eqref{CFLCProunds_binarias} specify that all decision variables are binary.

In this case, the reduction in the number of variables and constraints with respect to ${\normalfont \texttt{(CFLCP)}_{\texttt{CC}}}$ is quite significant. Constraints \eqref{CFLCPBrounds_allocation2} are no longer required; the families of constraints defined $\forall r$ now have fewer constraints because $r\in \J$ instead of $\I$; and the capacity constraints \eqref{CFLCPBrounds_uiguala0} and \eqref{CFLCPBrounds_uiguala1} correspond now to the smaller families \eqref{CFLCProunds_uiguala0} and \eqref{CFLCProunds_uiguala1}. The number of $u$-variables is reduced to only $|\J|^2$, and the number of $x$-variables is reduced from $|\I|^2 \cdot |\J|$ to $|\I| \cdot |\J|$. Furthermore, in the following proposition it is stated that the integrality constraints on the latter set of variables can be relaxed.

\begin{prop}
    The integrality constraints on the $x$-variables \eqref{CFLCProunds_xbinarias} can be relaxed in model ${\normalfont \texttt{(CFLCP)}_{\texttt{CC}}^\texttt{R}}$.
\end{prop}

\begin{proof}
For given integer values of the variables $y$ and $u$, we consider the same partition of the set $\J$ into $\J_1$, $\J_2$ and $\J_3$ as in Proposition~\ref{pro: 3}, where a plant is full if $\sum_{r \in \J} u_j^r = 1$.

With this partition, the structure of the model ${\normalfont \texttt{(CFLCP)}_{\texttt{CC}}^\texttt{R}}$ becomes equivalent to that of Proposition~\ref{pro: 3}, except for the presence of the additional constraints \eqref{CFLCProunds_preferencesFull}. These constraints consist of inequalities with a single nonzero coefficient equal to 1, and the right-hand side vector remains integral.

Adding such constraints to a totally unimodular matrix preserves total unimodularity. Hence, the full constraint matrix remains totally unimodular, and by \cite[Proposition~2.3]{wolsey1999integer}, the integrality constraints \eqref{CFLCProunds_xbinarias} can be relaxed.
\end{proof}

\subsection{Formulation for a maximum Pareto optimal matching for the CHA}
As an additional result, we note that formulation ${\normalfont \texttt{(CFLCP)}_{\texttt{CC}}^\texttt{R}}$ can be slightly modified to provide a Pareto optimal matching of maximum cardinality in an instance of CHA. In this setting, all plants are assumed to be open (we are only interested in the allocation problem) and the costs $f_j$ can be omitted. In addition, we maximize the number of clients allocated, so the objective weights $g_{ij}:=1$ and constraints \eqref{CFLCProunds_allocation} are relaxed. All in all, the model reads:
\begin{subequations} \label{CHArounds}
\begin{align} 
\texttt{(CHA)}_{\texttt{PO}}^{\texttt{R}} \quad \max_{\boldsymbol{x}, \boldsymbol{u}} \quad & \sum_{i\in \I} \sum_{j\in \J} x_{ij}  \label{CHArounds_OF1} \\
\text{s.t.} \quad & \sum_{j\in \J} x_{ij} \le 1, \quad \forall i\in \I, \label{CHArounds_allocation} \\
& \sum_{r\in \J} u_j^r \le 1, \quad \forall j\in\J, \label{CHArounds_seagotasiseabre} \\
& \sum_{j\in \J} u_j^1 \le 1, \label{CHArounds_seagotaunacadavez1} \\
& \sum_{j\in \J} u_j^r \le \sum_{j\in \J} u^{r-1}_j, \quad \forall r\in\J, r >1  \label{CHArounds_seagotaunacadavezresto} \\
&  \sum_{i\in \I} x_{ij} \le (c_j-1) + \sum_{r\in \J} u_j^r, \quad \forall j\in\J  \label{CHArounds_uiguala0} \\
& c_j\sum_{r\in \J} u_j^r \le \sum_{i\in \I} x_{ij}, \quad \forall j\in\J  \label{CHArounds_uiguala1} \\
&  x_{ij} + \sum_{r'=1}^{r-1} u_j^{r'} + \sum_{j'\in \J:\atop j' \prec_i j} u_{j'}^r \le 2, \quad \forall i\in \I,j,r\in\J, \label{CHArounds_preferencesFull} \\
&  1 - \sum_{r\in\J} u_j^r \le \sum_{j' \preceq_i j} x_{ij'}, \quad \forall i\in \I,j\in\J, \label{CHArounds_preferencesEmpty} \\
& x_{ij} \ge 0, \quad \forall i \in \I, j \in \J,  \label{CHArounds_xbinarias} \\
& u_j^r \in \{0,1\}, \quad \forall j,r \in \J.  \label{CHArounds_binarias}
\end{align}
\end{subequations}

%\subsection{Enhancements of model $\texttt{(CFLCP)}_{\texttt{CC}}^{\texttt{R}}$}
%We can obtain an upper bound on the number of plants that can end up full $R_{\max}$ by simply ordering the plants according to their capacity (from the smallest to the largest) and summing up on the capacity until we reach $|\I|$. 

%Valid inequalities:
%\begin{equation} \label{CFLCProunds_FeasibilityCut}
%\sum_{j\in \J} (c_j-1)y_j + \sum_{r\in\J} u_j^r \geq |\I|, \quad \forall i\in \I
%\end{equation}

\section{Computational experiments} \label{sec:compExp}   
In this section, we present the computational experiments carried out to evaluate and compare the proposed models. The objective is twofold: first, to compare the original formulations associated with each type of stability with their reduced counterparts; and second, to assess the performance quality of solutions of our models against those proposed in \cite{calvete2020}. All tests were performed using the commercial IP solver Xpress Mosel on a server equipped with an AMD Ryzen 9 7950X processor (3.4 GHz, 16 cores and 32 threads) and 4~\(\times\)~48~GB of DDR5 RAM (196~GB in total). The system allows up to 16 single-threaded tasks to run in parallel, each with 16~GB of allocated memory. For storage, it uses 2~\(\times\)~2~TB Western Digital Black SN770 SSDs configured in RAID~1 to ensure data reliability during execution.

\subsection{Comparison of each stable model with its reduced counterpart} %Computational experiment with synthetic instances}
We begin with a comparison between the original models and their reduced counterparts for the different notions of stability. To this end, several instances were randomly generated, with the number of clients, number of products, and capacity values fixed in advance. For each combination of these parameters, five different instances were considered, and the values shown in the tables correspond to the average of the results obtained across these five instances.

The instances were generated as follows: for the case of 50 clients, the number of facilities ranged between 5 and 10, and the capacity assigned to each facility was set between 12 and 20. For the case of 100 clients, the number of facilities varied among 5, 10, 15, and 20, with capacities ranging from 24 to 40. The cost of opening each facility was chosen as a random integer between 100 and 130, while the cost of assigning each client to each facility was selected as a random integer between 10 and 30. As for client preferences, a complete ranking over all facilities was generated randomly for each client.

For each type of stability, we present a table comparing the results obtained with the original model and its reduced version. Thus, the results for customer, pairwise, and cyclic-coalition stable models are depicted, respectively, in Tables \ref{tab:customer-results}-\ref{tab:cycliccoalition-results}. Each table includes the number of clients, the number of facilities, and the identical capacity assigned to each facility. These are followed by performance indicators obtained during the execution of the solver: the number of nodes explored in the branch-and-bound process; the linear relaxation bound computed at the root node; the total running time in seconds; and the final optimality gap. The time limit was set to one hour for all runs. In cases where no feasible integer solution was found within the time limit, the corresponding entry in the Gap column is marked with ``--”. The optimality gap is computed as \(100 \cdot \frac{UB - LB}{LB} \%\), where \(UB\) is the best integer solution value found within the time limit and \(LB\) is the final lower bound reported by the solver. When only a subset of the five instances is solved to optimality, the number of solved instances is indicated in parentheses next to the running time. This detail is omitted when all five instances are solved (which is reflected by a zero optimality gap) or when none are solved (which is evident from the maximum time of 3600 seconds).

\begin{table}%[H]
\centering
\begin{adjustbox}{max width=\textwidth}
\begin{tabular}{rrr  rrrr rrrr}
\toprule
Customers & Plants & Capacity & ${\normalfont\texttt{(CFLCP)}_{\texttt{CS}}}$ & & & & $\texttt{(CFLCP)}_{\texttt{CS}}^\texttt{R}$ & & & \\
\cmidrule(lr){4-7} \cmidrule(lr){8-11}
& & & Nodes & Time & LR Bound & Gap & Nodes & Time & LR Bound & Gap \\
\midrule
50&5&12&14641&7,5&1158,5&0&1&0,2&1158,5&0\\
50&5&20&3980,6&5,2&1030,6&0&88,6&0,4&1030,6&0\\[2mm]
50&10&12&74266&151&1112,8&0&2337,8&2,6&1112,8&0\\
50&10&20&13959,8&43,6&997,2&0&620,2&1,9&997,2&0\\[2mm]
100&5&24&536979&484,8&1828&0&5,8&0,5&1828&0\\
100&5&40&86035,8&91,3&1754,7&0&108,2&0,6&1754,7&0\\[2mm]
100&10&24&2474721&2850,4(2)&1752,3&3,68&4495,8&5,4&1752,3&0\\
100&10&40&232802,8&307,5&1711,1&0&1075&3,8&1711,1&0\\[2mm]
100&15&24&1475028,4&3600&1712&11,28&35526,2&39,3&1712&0\\
100&15&40&393874,6&1361,4&1672,9&0&3197&11&1672,9&0\\[2mm]
100&20&24&629432&3600&1705,6&15,71&168120,6&195,5&1705,6&0\\
100&20&40&495892&3176,5(1)&1661,1&5,46&7481&26,9&1661,1&0\\
\bottomrule
\end{tabular}
\end{adjustbox}
\caption{Results for the customer stable model \(\texttt{(CFLCP)}_{\texttt{CS}}\) and its reduced counterpart \(\texttt{(CFLCP)}_{\texttt{CS}}^\texttt{R}\). The column Customers indicates the number of clients, Plants the number of facilities, and Capacity the identical capacity assigned to each facility. The columns Nodes, Time, LR Bound, and Gap refer respectively to the number of nodes explored in the branch-and-bound tree, the total running time (in seconds), the linear relaxation bound at the root node, and the optimality gap.}
\label{tab:customer-results}
\end{table}
%\red{Entiendo que aquí has puesto la tabla ``no flotante'' y así no te refieres a ella, pero normalmente lo hacemos del otro modo. Te cambio la primera, y ligerametne el texto para que se entienda a qué tabla se refiere. El resto van igual.}

Regarding the comparison of models ${\normalfont\texttt{(CFLCP)}_{\texttt{CS}}}$ and $\texttt{(CFLCP)}_{\texttt{CS}}^\texttt{R}$, in Table \ref{tab:customer-results} we observe that the reduced model successfully solves all instances, while the original model fails in some cases. Whenever the total running time reaches 3600 seconds, none of the five instances were solved to optimality. If the time limit is not reached but a nonzero gap remains, it indicates that at least one instance was not solved within the time limit. On average, the reduced model explores significantly fewer nodes and consistently solves the instances in under 4 minutes, unlike the original model, which struggles to do so in several configurations. These empirical results provide solid support for discarding the original formulation in favor of the reduced model, whose superior performance is evident across all tested instances.

Interestingly, both models yield the same linear relaxation bound at the root node. This can be explained by the fact that any solution to the relaxation of the reduced model can be transformed into a feasible solution for the relaxation of the original model by suitably increasing the values of the auxiliary variables \(u_j\) up to their upper bounds—that is, so that constraint \eqref{CFLCP_CSR_capacity0} is satisfied with equality. Since these variables do not appear in the objective function, this adjustment has no effect on the objective value. As a result, all constraints in the original model are satisfied, and the same bound is preserved.

In general, when two of the three parameters (number of clients, number of facilities, and facility capacity) are fixed, the model tends to become more complex as the number of clients or facilities increases, due to the larger number of variables and constraints involved. Conversely, increasing the facility capacity tends to simplify the problem, as higher capacities reduce the combinatorial effort required to avoid exceeding them.

\begin{table}%[H]
\centering
\begin{adjustbox}{max width=\textwidth}
\begin{tabular}{rrr rrrr rrrr}
\toprule
Customers & Plants & Capacity & ${\normalfont\texttt{(CFLCP)}_{\texttt{PW}}}$ & & & & $\texttt{(CFLCP)}_{\texttt{PW}}^\texttt{R}$ & & & \\
\cmidrule(lr){4-7} \cmidrule(lr){8-11}
& & & Nodes & Time & LR Bound & Gap & Nodes & Time & LR Bound & Gap \\
\midrule
50&5&12&5158,6&7,8&1197,6&0&51,8&0,3&1212,2&0 \\
50&5&20&158,2&3&1033,5&0&104,6&0,3&1036&0 \\[2mm]
50&10&12&68717,8&346,5&1118,1&0&31979,8&16,3&1124,1&0 \\
50&10&20&1508,6&14,2&997,2&0&1344,2&3,1&997,3&0 \\[2mm]
100&5&24&67063,2&262,8&1889,9&0&169,4&0,6&1917,4&0 \\
100&5&40&3309&21,9&1754,1&0&196,6&0,6&1754,7&0 \\[2mm]
100&10&24&127253,6&3600&1752,8&21,1&246916,8&419,7&1756,1&0 \\
100&10&40&18346,6&184,4&1701,9&0&2735&12&1711,1&0 \\[2mm]
100&15&24&23150,2&3600&1712,1&37,84&409771,6&3600&1713,3&16,3 \\
100&15&40&41507,6&3454(1)&1665,6&16,42&9162,6&83,3&1672,9&0 \\[2mm]
100&20&24&6583,4&3600&1705,7&--&139989,4&3600&1705,9&23,3 \\
100&20&40&6325,4&3600&1654,7&29,34&14292,2&555,8&1661,1&0 \\
\bottomrule
\end{tabular}
\end{adjustbox}
\caption{Results for the pairwise stable model \(\texttt{(CFLCP)}_{\texttt{PW}}\) and its reduced counterpart \(\texttt{(CFLCP)}_{\texttt{PW}}^\texttt{R}\). The column Customers indicates the number of clients, Plants the number of facilities, and Capacity the identical capacity assigned to each facility. The columns Nodes, Time, LR Bound, and Gap refer respectively to the number of nodes explored in the branch-and-bound tree, the total running time (in seconds), the linear relaxation bound at the root node, and the optimality gap.}
\label{tab:pairwise-results}
\end{table}

The performance of models ${\normalfont\texttt{(CFLCP)}_{\texttt{PW}}}$ and $\texttt{(CFLCP)}_{\texttt{PW}}^\texttt{R}$ is depicted in Table \ref{tab:pairwise-results}. We observe that in this case neither model is able to solve all instances. In particular, neither of the two formulations manages to solve any of the instances with 100 clients, 15 or 20 facilities, and facility capacity equal to 24 within the time limit. However, the reduced model succeeds in solving the remaining instances, whereas the original formulation fails to solve several additional configurations.

A behavior similar to the one observed in the customer stable case arises here as well: the reduced model consistently explores fewer nodes and achieves shorter running times. Moreover, unlike in the previous case, the linear relaxation bounds at the root node differ between the two formulations, with the reduced model providing tighter bounds. This additional advantage further supports the empirical preference for the reduced model, which outperforms the original formulation in terms of both time and solvability. These results justify the dismissal of the original model in favor of its reduced counterpart. 

Regarding parameter influence, the same trend holds: the instances that combine a higher number of clients and facilities with lower facility capacity are the most computationally demanding. In fact, the hardest instances—those unsolved by both models—correspond to the most extreme parameter combinations within our set of instances, or to configurations that are nearly as extreme.

\begin{table}%[H]
\centering
\begin{adjustbox}{max width=\textwidth}
\begin{tabular}{rrr  rrrr rrrr}
\toprule
Customers & Plants & Capacity & ${\normalfont\texttt{(CFLCP)}_{\texttt{CC}}}$ & & & & $\texttt{(CFLCP)}_{\texttt{CC}}^\texttt{R}$ & & & \\
\cmidrule(lr){4-7} \cmidrule(lr){8-11}
& & & Nodes & Time & LR Bound & Gap & Nodes & Time & LR Bound & Gap \\
\midrule
50&5&12&10926,4&3600&1162,8&16,2&35,8&0,9&1161,1&0 \\
50&5&20&22458&3600&1031,8&12,5&98,6&1,1&1030,6&0 \\[2mm]
50&10&12&23,2&3600&1114,7&27,8&19539,8&58,1&1112,8&0 \\
50&10&20&304,8&3600&997,8&17,2&985,4&9,3&997,2&0 \\[2mm]
100&5&24&3&3600&1832,3&--&145&2,1&1833,3&0 \\
100&5&40&364,8&3600&1756,7&24,5&141,4&1,4&1754,7&0 \\[2mm]
100&10&24&0&3600&1754,6&--&24029,4&166,9&1752,3&0 \\
100&10&40&0&3600&1711,7&--&1199,8&24,3&1711,1&0 \\[2mm]
100&15&24&0&3600&1713,6&--&157516,4&2290,5(4)&1712&1,2 \\
100&15&40&0&3600&1673,2&--&3325,6&167,7&1672,9&0 \\[2mm]
100&20&24&0&3600&1706,9&--&230371,6&3484,5(1)&1705,6&10,9 \\
100&20&40&0&3600&1661,4&--&5187&617,1&1661,1&0 \\
\bottomrule
\end{tabular}
\end{adjustbox}
\caption{Results for the cyclic-coalition stable model \(\texttt{(CFLCP)}_{\texttt{CC}}\) and its reduced counterpart \(\texttt{(CFLCP)}_{\texttt{CC}}^\texttt{R}\). The column Customers indicates the number of clients, Plants the number of facilities, and Capacity the identical capacity assigned to each facility. The columns Nodes, Time, LR Bound, and Gap refer respectively to the number of nodes explored in the branch-and-bound tree, the total running time (in seconds), the linear relaxation bound at the root node, and the optimality gap.}
\label{tab:cycliccoalition-results}
\end{table}

A detailed comparison between models ${\normalfont\texttt{(CFLCP)}_{\texttt{CC}}}$ and $\texttt{(CFLCP)}_{\texttt{CC}}^\texttt{R}$ is presented in Table~\ref{tab:cycliccoalition-results}. In this case, the initial model is unable to solve any of the generated instances: it only finds feasible solutions for some of them, but with a very high gap. In contrast, the reduced model manages to solve almost all the provided instances, except for two specific cases in which it fails to solve some of them. Moreover, the reduced model consistently exhibits a much lower number of explored nodes and significantly shorter running times compared to the original formulation. This reinforces the pattern already observed in the other stability variants, further highlighting the computational advantages of the reduced approach in terms of both efficiency and reliability. In this case, the improvement is more significant than in previous cases because there is a noticeable reduction in the number of variables: model ${\normalfont\texttt{(CFLCP)}_{\texttt{PW}}}$ makes use of binary variables $x^r_{ij}$ (a total of $|\I|^2\cdot|\J|$) to account for customers' decision choice, while model $\texttt{(CFLCP)}_{\texttt{PW}}^\texttt{R}$ uses the much smaller set of continuous variables $x_{ij}$ (only $|\I|\cdot|\J|$).

A similar effect is observed regarding the influence of each parameter. When two of the three parameters (number of clients, number of facilities, and facility capacity) are fixed, increasing the number of clients or facilities tends to make the problem more complex, as it enlarges the solution space and the number of constraints to manage. In contrast, increasing the facility capacity generally eases the computational burden, since higher capacities reduce the number of combinations that must be evaluated to ensure feasibility.

Finally, note that more realistic configurations where the list of preferences are incomplete would be easier to tackle for the models. In fact, these instances are naturally sparse and, since every customer needs to be assigned to a plant, they force some sets of plants to open and reduce the combinatorics. These advantages allow to solve instances with a larger number of customers and plants within the same time limit. However, we decided to test only instances with a complete list of preferences because our interest focuses on the comparison between the two formulations presented for each setting.

%\red{Si se te ocurren párrafos que añadir aquí como el anterior qeu he puesto yo, son bienvenidos.}

%\red{He cambiado los nombres de las secciones (aunqeu los puse yo) porque no me gustaba que se citara un paper en un nombre de sección y eso. Si quieres modificarlos, dale, que no se me ocurre nada decente ahora mismo...}

\subsection{Comparison of the solution quality provided by each setting}

In this section, we compare the quality of the solutions obtained under our different settings with the optimal values reported in the computational experiment presented in \cite{calvete2020}, where model \eqref{BMCalvete} was used. The procedure for generating the instances is detailed in that work. While the capacity of each plant, as well as the cost and preference values, are randomly generated, the structural parameters of the instances---namely, the number of clients and the number of plants---are fixed and known.

It is important to note that we do not aim to evaluate or compare the performance of the model proposed in \cite{calvete2020}. Instead, we use the same set of benchmark instances to assess how the optimal objective values obtained by their model compare to the results produced by our alternative solution settings.

\begin{table}
\centering
\begin{adjustbox}{max width=\textwidth}
\begin{tabular}{c r rrr rrr rrr}
\toprule
Instance & $\texttt{(CGICC)}$ & $\texttt{(CFLCP)}_{\texttt{CS}}^\texttt{R}$ & & & $\texttt{(CFLCP)}_{\texttt{PW}}^\texttt{R}$ & & & $\texttt{(CFLCP)}_{\texttt{CC}}^\texttt{R}$& & \\
\cmidrule(lr){2-2} \cmidrule(lr){3-5} \cmidrule(lr){6-8}  \cmidrule(lr){9-11}
& Obj & Obj & Time & Gap & Obj & Time & Gap & Obj & Time & Gap \\
\midrule
$P_1$&\textbf{18592}&\textbf{10016}&0,6&0&\textbf{15017}&4,8&0&\textbf{15335}&424,4&0 \\
$P_2$&\textbf{17658}&\textbf{9327}&0,4&0&\textbf{14404}&5,5&0&\textbf{14598}&86,6&0 \\
$P_3$&\textbf{19058}&\textbf{10527}&0,3&0&\textbf{15604}&6,6&0&\textbf{15798}&143,8&0 \\
$P_4$&\textbf{20442}&\textbf{11727}&0,4&0&\textbf{16736}&5,5&0&\textbf{16998}&1508,8&0 \\[2mm]
$P_5$&\textbf{18552}&\textbf{13832}&1,2&0&\textbf{15347}&16,3&0&\textbf{16050}&275&0 \\
$P_6$&\textbf{17806}&\textbf{12924}&1&0&\textbf{14601}&16&0&\textbf{15304}&402,7&0 \\
$P_7$&\textbf{19206}&\textbf{14324}&0,7&0&\textbf{16001}&22,8&0&\textbf{16704}&142,8&0 \\
$P_8$&\textbf{20606}&\textbf{15724}&1,3&0&\textbf{17401}&28&0&\textbf{18104}&167,1&0 \\[2mm]
$P_9$&\textbf{17651}&\textbf{14072}&1,2&0&\textbf{15141}&8,7&0&\textbf{15717}&1246,4&0 \\
$P_{10}$&\textbf{17146}&\textbf{13616}&1,5&0&\textbf{14685}&11,5&0&\textbf{15261}&263,3&0 \\
$P_{11}$&\textbf{18146}&\textbf{14616}&1,3&0&\textbf{15685}&9,1&0&\textbf{16261}&971,2&0 \\
$P_{12}$&\textbf{19146}&\textbf{15616}&1,2&0&\textbf{16685}&11,1&0&\textbf{17261}&59,8&0  \\[2mm]
$P_{13}$&\textbf{17745}&\textbf{8696}&0,8&0&14094&7200&5,09&16665&7200&85,11 \\
$P_{14}$&\textbf{16720}&\textbf{7922}&0,9&0&\textbf{13210}&3218,7&0&16569&7200&104,8 \\
$P_{15}$&\textbf{18120}&\textbf{9322}&1&0&\textbf{14410}&6629,5&0&18101&7200&87,54 \\
$P_{16}$&\textbf{19427}&\textbf{10717}&1,4&0&\textbf{15610}&2635,5&0&17623&7200&64,16 \\[2mm]
$P_{17}$&\textbf{17613}&\textbf{13511}&186,1&0&15582&7200&16,37&16593&7200&56,18 \\
$P_{18}$&\textbf{16718}&\textbf{12554}&497,2&0&14553&7200&14&15239&7200&44,91 \\
$P_{19}$&\textbf{18118}&\textbf{13954}&367&0&15814&7200&14,56&17448&7200&54,73 \\
$P_{20}$&\textbf{19518}&\textbf{15354}&302,2&0&17094&7200&13,03&21954&7200&98,76 \\[2mm]
$P_{21}$&\textbf{17253}&\textbf{13442}&173,1&0&15117&7200&10,27&15604&7200&31,78 \\
$P_{22}$&\textbf{16407}&\textbf{12433}&123,7&0&14024&7200&7,59&14906&7200&34,12 \\
$P_{23}$&\textbf{17607}&\textbf{13633}&177,6&0&15224&7200&7,86&18148&7200&82,34 \\
$P_{24}$&\textbf{18807}&\textbf{14833}&153,2&0&16424&7200&0,99&19277&7200&74,8 \\

\bottomrule
\end{tabular}
\end{adjustbox}
\caption{Results for the reduced models \(\texttt{(CFLCP)}_{\texttt{CS}}^\texttt{R}\), \(\texttt{(CFLCP)}_{\texttt{PW}}^\texttt{R}\), and \(\texttt{(CFLCP)}_{\texttt{CC}}^\texttt{R}\), compared with the original model proposed by \cite{calvete2020}, referred to as $\texttt{(CGICC)}$. The column Instance indicates the instance identifier. Under $\texttt{(CGICC)}$, the column Obj corresponds to the objective value reported in their study. For the reduced models, the columns Obj, Time, and Gap respectively indicate the best objective value found, the total running time in seconds, and the final optimality gap.}
%\caption{Results for the reduced models \(\texttt{(CFLCP)}_{\texttt{CS}}^\texttt{R}\), \(\texttt{(CFLCP)}_{\texttt{PW}}^\texttt{R}\), and \(\texttt{(CFLCP)}_{\texttt{CC}}^\texttt{R}\), compared with the original model proposed by \cite{calvete2020}, referred to as $\texttt{(CGICC)}$. The column Instance indicates the instance identifier. Under $\texttt{(CGICC)}$, the column Obj corresponds to the objective value reported in their study. For the reduced models, the columns Obj, Time, and Gap respectively indicate the best objective value found, the total running time in seconds, and the final optimality gap. \red{Yo creo que en esta tabla pondría en negrita algo distinto: para cada P, la solución (de entre la de Calvete y nuestra "strongly stable") con mejor valor óptimo. Una idea final para mí era comparar nuestro modelo strongly con el suyo, porque al final ``hacen lo mismo'' pero nosotros buscamos entre más soluciones y por eso tenemos mejor valor objetivo (esa es la idea que yo quiero vender un poco). }}
\label{tab:calvete-results}
\end{table}

For consistency, we restrict our comparison to the first 24 instances considered in their study. All of them involve 50 customers. The first 12 instances use 10 plants, while the remaining 12 use 20. The rest of the original dataset was not included due to time or memory limitations encountered when solving the pairwise and cyclic-coalition stable models. In our experiments, we set a time limit of 2 hours per instance. Nonetheless, for some of the larger cases, we were able to obtain high-quality incumbent solutions, which are comparable to or even better than those reported in \cite{calvete2020}.

The performance of the reduced models $\texttt{(CFLCP)}_{\texttt{CS}}^\texttt{R}$, $\texttt{(CFLCP)}_{\texttt{PW}}^\texttt{R}$, and $\texttt{(CFLCP)}_{\texttt{CC}}^\texttt{R}$ is summarized in Table~\ref{tab:calvete-results}, where we compare their results with those obtained using the original model proposed by \cite{calvete2020}, which we refer to as $\texttt{(CGICC)}$. For each instance, we report the objective value obtained by $\texttt{(CGICC)}$, and for the reduced models we include the objective value, the computational time (in seconds), and the final optimality gap.

We observe, first, that model $\texttt{(CFLCP)}_{\texttt{CS}}^\texttt{R}$ is able to solve all 24 instances within the imposed time limit of two hours. In contrast, model $\texttt{(CFLCP)}_{\texttt{PW}}^\texttt{R}$ solves 15 of the first 16 instances, while $\texttt{(CFLCP)}_{\texttt{CC}}^\texttt{R}$ successfully solves only the first 12. As expected—since it is a less restrictive formulation—$\texttt{(CFLCP)}_{\texttt{CS}}^\texttt{R}$ consistently finds optimal solutions with lower objective values than those obtained with the original model $\texttt{(CGICC)}$. Moreover, it achieves this in under two seconds for the first 16 instances and in less than five minutes for the remaining ones.

A similar pattern is observed for $\texttt{(CFLCP)}_{\texttt{PW}}^\texttt{R}$. Although not all instances are solved to optimality within the time limit, a feasible solution is found in every case, and all of them improve upon the objective values reported by $\texttt{(CGICC)}$. This model solves the first 12 instances in under 30 seconds, while the remaining ones either exceed the time limit or require significantly more effort.

In the case of model $\texttt{(CFLCP)}_{\texttt{CC}}^\texttt{R}$, we obtain optimal solutions for the first 12 instances, all of which improve upon the original model. Even in the instances where optimality is not reached, several feasible solutions still yield better objective values than those reported in \cite{calvete2020}. However, this model requires significantly more time to reach such solutions. These results reinforce the idea that the cyclic-coalition stable formulation is less restrictive than the original model, allowing for a broader set of feasible allocations. As a consequence, it can potentially capture more efficient configurations that still satisfy a meaningful notion of stability, although this comes at the cost of greater computational effort.

It is worth highlighting a clear pattern in the results: as the stability concept becomes more restrictive---from customer stability to pairwise stability, and finally to cyclic-coalition stability---the objective function value systematically increases. This behavior confirms that many blocking structures involving more than two agents do appear in practice, and they are effectively eliminated only under the strongest stability notion. Hence, each increase in model restrictiveness translates into higher total cost, as the solution space becomes progressively constrained.
\begin{table}
\centering
\begin{adjustbox}{max width=\textwidth}
\begin{tabular}{c r r r}
\toprule
Instance & $\texttt{(CGICC)}$ Obj & $\texttt{(CFLCP)}_{\texttt{CC}}^\texttt{R}$ Obj & Relative Diff. (\%) \\
\midrule
$P_1$&18592&\textbf{15335}&-17,52 \\
$P_2$&17658&\textbf{14598}&-17,33 \\
$P_3$&19058&\textbf{15798}&-17,11 \\
$P_4$&20442&\textbf{16998}&-16,85 \\[2mm]
$P_5$&18552&\textbf{16050}&-13,49 \\
$P_6$&17806&\textbf{15304}&-14,05 \\
$P_7$&19206&\textbf{16704}&-13,03 \\
$P_8$&20606&\textbf{18104}&-12,14\\[2mm]
$P_9$&17651&\textbf{15717}&-10,96 \\
$P_{10}$&17146&\textbf{15261}&-10,99 \\
$P_{11}$&18146&\textbf{16261}&-10,39 \\
$P_{12}$&19146&\textbf{17261}&-9,85 \\[2mm]
$P_{13}$&17745&\textbf{16665}&-6,09 \\
$P_{14}$&16720&\textbf{16569}&-0,90 \\
$P_{15}$&18120&\textbf{18101}&-0,10 \\
$P_{16}$&19427&\textbf{17623}&-9,29 \\[2mm]
$P_{17}$&17613&\textbf{16593}&-5,79 \\
$P_{18}$&16718&\textbf{15239}&-8,85 \\
$P_{19}$&18118&\textbf{17448}&-3,70 \\
$P_{20}$&\textbf{19518}&21954&12,48 \\[2mm]
$P_{21}$&17253&\textbf{15604}&-9,56 \\
$P_{22}$&16407&\textbf{14906}&-9,15 \\
$P_{23}$&\textbf{17607}&18148&3,07 \\
$P_{24}$&\textbf{18807}&19277&2,50 \\
\bottomrule
\end{tabular}
\end{adjustbox}
\caption{Comparison between the original model proposed by \cite{calvete2020} ($\texttt{(CGICC)}$) and our cyclic-coalition stable formulation ($\texttt{(CFLCP)}_{\texttt{CC}}^\texttt{R}$). For each instance, we report the objective values obtained with both models, as well as the relative percentage change. Negative values indicate improvements (lower total cost) achieved by our model, while positive values reflect cases where $\texttt{(CGICC)}$ performs better.}

\label{tab:cc-vs-calvete}
\end{table}
It is also worth noting that the original model $\texttt{(CGICC)}$ demonstrates a certain degree of robustness, as it provides optimal solutions for all tested instances under its own formulation. However, this model imposes a very restrictive notion of stability, which considerably narrows the solution space and may exclude other desirable allocations that remain stable under broader definitions. In contrast, the reduced models introduced in this work not only yield improved objective values, but also offer the flexibility to accommodate varying degrees of stability, allowing the decision-maker to balance efficiency and fairness according to the specific context.

Among these, model $\texttt{(CFLCP)}{\texttt{CC}}^\texttt{R}$ deserves particular attention due to its conceptual alignment with $\texttt{(CGICC)}$. Both aim to deliver allocations that are globally satisfactory for all clients; however, while $\texttt{(CGICC)}$ identifies a few feasible solutions that satisfy its stringent constraints, $\texttt{(CFLCP)}{\texttt{CC}}^\texttt{R}$ explores the full space of cyclic-coalition stable allocations. This broader scope avoids arbitrary prioritization among clients and enables the discovery of more cost-effective solutions.

To better illustrate this comparison, Table~\ref{tab:cc-vs-calvete} reports the objective values achieved by both models across the 24 instances, along with the relative percentage difference. The first 12 instances were solved to optimality under our cyclic-coalition stability model, while the remaining ones reached the 2-hour time limit. In most cases, $\texttt{(CFLCP)}_{\texttt{CC}}^\texttt{R}$ obtains strictly better objective values, underscoring its ability to identify more efficient and globally stable solutions.

Importantly, in the few instances where our model yields a worse objective value, this is necessarily due to the fact that optimality was not reached within the time limit. Since any feasible solution to the original model is also feasible for ours, no degradation in solution quality can occur unless the solver is interrupted before convergence. The relative improvement is especially notable in several instances, suggesting that the original model's restrictive framework may preclude access to high-quality allocations that remain stable under a broader, yet still meaningful, stability concept.

\section{Conclusions and future research} \label{sec:conclusion}
In this work, we have provided new definitions of stability and derived associated mixed-integer linear optimization models to solve the CFLCP in the global preference maximization framework. 

The cyclic-coalition stable setting poses several challenges that we plan to address in future research. In particular, when ties arise in the customers’ preference lists, the models proposed in this article are no longer directly applicable. In such cases, new formulations are required that incorporate additional structural properties and draw on results from the matching theory literature. Exploring these extensions would not only allow for a better understanding of stability under customer preferences with indifference, but could also lead to novel algorithms capable of handling more complex and realistic preference structures.

In the deterministic version of the problem solved in this article, it is assumed that all customers can be allocated \emph{simultaneously} in the most favorable way for the decision maker. This corresponds to an \emph{optimistic} solution. However, in many practical applications of this and related problems \citep[see, e.g.,][]{hassin2024}, customers arrive sequentially. This motivates considering extensions of the model that yield more robust solutions. A \emph{pessimistic approach}, for instance, determines the optimal plant locations under the assumption that customers are allocated in the most costly way, leading to a bilevel min–max formulation. Other intermediate approaches could also be explored, such as computing mean-cost or median-cost allocations when customer arrivals are assumed to be random.

\section*{Acknowledgements} 

All the authors acknowledge financial support through the grant RED2022-134149-T funded by MICIU/AEI/10.13039/501100011033
(Thematic Network on Location Science and Related Problems).

\bibliographystyle{elsarticle-harv} 
\bibliography{references}

\begin{thebibliography}{23}
\expandafter\ifx\csname natexlab\endcsname\relax\def\natexlab#1{#1}\fi
\providecommand{\url}[1]{\texttt{#1}}
\providecommand{\href}[2]{#2}
\providecommand{\path}[1]{#1}
\providecommand{\DOIprefix}{doi:}
\providecommand{\ArXivprefix}{arXiv:}
\providecommand{\URLprefix}{URL: }
\providecommand{\Pubmedprefix}{pmid:}
\providecommand{\doi}[1]{\href{http://dx.doi.org/#1}{\path{#1}}}
\providecommand{\Pubmed}[1]{\href{pmid:#1}{\path{#1}}}
\providecommand{\bibinfo}[2]{#2}
\ifx\xfnm\relax \def\xfnm[#1]{\unskip,\space#1}\fi
%Type = Article
\bibitem[{Avella and Boccia(2009)}]{avella2009}
\bibinfo{author}{Avella, P.}, \bibinfo{author}{Boccia, M.},
  \bibinfo{year}{2009}.
\newblock \bibinfo{title}{A cutting plane algorithm for the capacitated
  facility location problem}.
\newblock \bibinfo{journal}{Computational Optimization and Applications}
  \bibinfo{volume}{43}, \bibinfo{pages}{39--65}.
%Type = Inproceedings
\bibitem[{B{\"u}sing et~al.(2022)B{\"u}sing, Gersing and Wrede}]{busing2022}
\bibinfo{author}{B{\"u}sing, C.}, \bibinfo{author}{Gersing, T.},
  \bibinfo{author}{Wrede, S.}, \bibinfo{year}{2022}.
\newblock \bibinfo{title}{Analysing the complexity of facility location
  problems with capacities, revenues, and closest assignments}, in:
  \bibinfo{booktitle}{INOC}, pp. \bibinfo{pages}{1--6}.
%Type = Article
\bibitem[{B{\"u}sing et~al.(2025)B{\"u}sing, Leitner and Wrede}]{busing2025}
\bibinfo{author}{B{\"u}sing, C.}, \bibinfo{author}{Leitner, M.},
  \bibinfo{author}{Wrede, S.}, \bibinfo{year}{2025}.
\newblock \bibinfo{title}{Cover-based inequalities for the single-source
  capacitated facility location problem with customer preferences}.
\newblock \bibinfo{journal}{Computers and Operations Research}
  \bibinfo{volume}{182}.
%Type = Article
\bibitem[{Cabezas and Garc{\'\i}a(2023)}]{cabezas2023}
\bibinfo{author}{Cabezas, X.}, \bibinfo{author}{Garc{\'\i}a, S.},
  \bibinfo{year}{2023}.
\newblock \bibinfo{title}{A semi-lagrangian relaxation heuristic algorithm for
  the simple plant location problem with order}.
\newblock \bibinfo{journal}{Journal of the Operational Research Society}
  \bibinfo{volume}{74}, \bibinfo{pages}{2391--2402}.
%Type = Article
\bibitem[{Calvete et~al.(2020)Calvete, Gal{\'e}, Iranzo, Camacho-Vallejo and
  Casas-Ram{\'\i}rez}]{calvete2020}
\bibinfo{author}{Calvete, H.I.}, \bibinfo{author}{Gal{\'e}, C.},
  \bibinfo{author}{Iranzo, J.A.}, \bibinfo{author}{Camacho-Vallejo, J.F.},
  \bibinfo{author}{Casas-Ram{\'\i}rez, M.S.}, \bibinfo{year}{2020}.
\newblock \bibinfo{title}{A matheuristic for solving the bilevel approach of
  the facility location problem with cardinality constraints and preferences}.
\newblock \bibinfo{journal}{Computers \& Operations Research}
  \bibinfo{volume}{124}, \bibinfo{pages}{105066}.
%Type = Article
\bibitem[{C{\'a}novas et~al.(2007)C{\'a}novas, Garc{\'\i}a, Labb{\'e} and
  Mar{\'\i}n}]{canovas2007}
\bibinfo{author}{C{\'a}novas, L.}, \bibinfo{author}{Garc{\'\i}a, S.},
  \bibinfo{author}{Labb{\'e}, M.}, \bibinfo{author}{Mar{\'\i}n, A.},
  \bibinfo{year}{2007}.
\newblock \bibinfo{title}{A strengthened formulation for the simple plant
  location problem with order}.
\newblock \bibinfo{journal}{Operations Research Letters} \bibinfo{volume}{35},
  \bibinfo{pages}{141--150}.
%Type = Article
\bibitem[{Casas-Ram{\'\i}rez et~al.(2018)Casas-Ram{\'\i}rez, Camacho-Vallejo
  and Mart{\'\i}nez-Salazar}]{casas2018}
\bibinfo{author}{Casas-Ram{\'\i}rez, M.S.}, \bibinfo{author}{Camacho-Vallejo,
  J.F.}, \bibinfo{author}{Mart{\'\i}nez-Salazar, I.A.}, \bibinfo{year}{2018}.
\newblock \bibinfo{title}{Approximating solutions to a bilevel capacitated
  facility location problem with customer's patronization toward a list of
  preferences}.
\newblock \bibinfo{journal}{Applied Mathematics and Computation}
  \bibinfo{volume}{319}, \bibinfo{pages}{369--386}.
%Type = Article
\bibitem[{Delorme et~al.(2019)Delorme, Garc{\'\i}a, Gondzio, Kalcsics, Manlove
  and Pettersson}]{delorme2019}
\bibinfo{author}{Delorme, M.}, \bibinfo{author}{Garc{\'\i}a, S.},
  \bibinfo{author}{Gondzio, J.}, \bibinfo{author}{Kalcsics, J.},
  \bibinfo{author}{Manlove, D.}, \bibinfo{author}{Pettersson, W.},
  \bibinfo{year}{2019}.
\newblock \bibinfo{title}{Mathematical models for stable matching problems with
  ties and incomplete lists}.
\newblock \bibinfo{journal}{European Journal of Operational Research}
  \bibinfo{volume}{277}, \bibinfo{pages}{426--441}.
%Type = Article
\bibitem[{Espejo et~al.(2012)Espejo, Mar{\'\i}n and
  Rodr{\'\i}guez-Ch{\'\i}a}]{espejo2012}
\bibinfo{author}{Espejo, I.}, \bibinfo{author}{Mar{\'\i}n, A.},
  \bibinfo{author}{Rodr{\'\i}guez-Ch{\'\i}a, A.M.}, \bibinfo{year}{2012}.
\newblock \bibinfo{title}{Closest assignment constraints in discrete location
  problems}.
\newblock \bibinfo{journal}{European Journal of Operational Research}
  \bibinfo{volume}{219}, \bibinfo{pages}{49--58}.
%Type = Incollection
\bibitem[{Fern{\'a}ndez and Landete(2015)}]{fernandez2015}
\bibinfo{author}{Fern{\'a}ndez, E.}, \bibinfo{author}{Landete, M.},
  \bibinfo{year}{2015}.
\newblock \bibinfo{title}{Fixed-charge facility location problems}, in:
  \bibinfo{booktitle}{Location science}. \bibinfo{publisher}{Springer}, pp.
  \bibinfo{pages}{47--77}.
%Type = Article
\bibitem[{Fischetti et~al.(2016)Fischetti, Ljubi{\'c} and
  Sinnl}]{fischetti2016}
\bibinfo{author}{Fischetti, M.}, \bibinfo{author}{Ljubi{\'c}, I.},
  \bibinfo{author}{Sinnl, M.}, \bibinfo{year}{2016}.
\newblock \bibinfo{title}{Benders decomposition without separability: A
  computational study for capacitated facility location problems}.
\newblock \bibinfo{journal}{European Journal of Operational Research}
  \bibinfo{volume}{253}, \bibinfo{pages}{557--569}.
%Type = Article
\bibitem[{G{\"o}rtz and Klose(2012)}]{gortz2012}
\bibinfo{author}{G{\"o}rtz, S.}, \bibinfo{author}{Klose, A.},
  \bibinfo{year}{2012}.
\newblock \bibinfo{title}{A simple but usually fast branch-and-bound algorithm
  for the capacitated facility location problem}.
\newblock \bibinfo{journal}{INFORMS Journal on Computing} \bibinfo{volume}{24},
  \bibinfo{pages}{597--610}.
%Type = Article
\bibitem[{Hanjoul and Peeters(1987)}]{hanjoul1987}
\bibinfo{author}{Hanjoul, P.}, \bibinfo{author}{Peeters, D.},
  \bibinfo{year}{1987}.
\newblock \bibinfo{title}{A facility location problem with clients' preference
  orderings}.
\newblock \bibinfo{journal}{Regional Science and Urban Economics}
  \bibinfo{volume}{17}, \bibinfo{pages}{451--473}.
%Type = Article
\bibitem[{Hansen et~al.(2004)Hansen, Kochetov and Mladenovic}]{hansen2004}
\bibinfo{author}{Hansen, P.}, \bibinfo{author}{Kochetov, Y.},
  \bibinfo{author}{Mladenovic, N.}, \bibinfo{year}{2004}.
\newblock \bibinfo{title}{Lower bounds for the uncapacitated facility location
  problem with user preferences.}
\newblock \bibinfo{journal}{Groupe d'études et de recherche en analyse des
  décisions, HEC Montréal.} , \bibinfo{pages}{6}.
%Type = Article
\bibitem[{Hassin and Puerto(2024)}]{hassin2024}
\bibinfo{author}{Hassin, R.}, \bibinfo{author}{Puerto, J.},
  \bibinfo{year}{2024}.
\newblock \bibinfo{title}{Pricing heterogeneous products to heterogeneous
  customers who buy sequentially}.
\newblock \bibinfo{journal}{Annals of Operations Research}
  \bibinfo{volume}{340}, \bibinfo{pages}{863--890}.
%Type = Inproceedings
\bibitem[{Kwanashie and Manlove(2014)}]{kwanashie2014}
\bibinfo{author}{Kwanashie, A.}, \bibinfo{author}{Manlove, D.F.},
  \bibinfo{year}{2014}.
\newblock \bibinfo{title}{An integer programming approach to the
  hospitals/residents problem with ties}, in: \bibinfo{booktitle}{Operations
  Research Proceedings 2013: Selected Papers of the International Conference on
  Operations Research, OR2013, organized by the German Operations Research
  Society (GOR), the Dutch Society of Operations Research (NGB) and Erasmus
  University Rotterdam, September 3-6, 2013}, \bibinfo{organization}{Springer}.
  pp. \bibinfo{pages}{263--269}.
%Type = Book
\bibitem[{Manlove(2013)}]{manlove2013}
\bibinfo{author}{Manlove, D.}, \bibinfo{year}{2013}.
\newblock \bibinfo{title}{Algorithmics of matching under preferences}.
  volume~\bibinfo{volume}{2}.
\newblock \bibinfo{publisher}{World Scientific}.
%Type = Article
\bibitem[{Polino et~al.(2024)Polino, Camacho-Vallejo and Villegas}]{polino2024}
\bibinfo{author}{Polino, S.}, \bibinfo{author}{Camacho-Vallejo, J.F.},
  \bibinfo{author}{Villegas, J.G.}, \bibinfo{year}{2024}.
\newblock \bibinfo{title}{A facility location problem for extracurricular
  workshop planning: bi-level model and metaheuristics}.
\newblock \bibinfo{journal}{International Transactions in Operational Research}
  \bibinfo{volume}{31}, \bibinfo{pages}{4025--4067}.
%Type = Phdthesis
\bibitem[{Sng(2008)}]{sng2008}
\bibinfo{author}{Sng, C.T.S.}, \bibinfo{year}{2008}.
\newblock \bibinfo{title}{Efficient algorithms for bipartite matching problems
  with preferences}.
\newblock Ph.D. thesis. University of Glasgow.
%Type = Article
\bibitem[{Vasilyev et~al.(2013)Vasilyev, Klimentova and Boccia}]{vasilyev2013}
\bibinfo{author}{Vasilyev, I.}, \bibinfo{author}{Klimentova, X.},
  \bibinfo{author}{Boccia, M.}, \bibinfo{year}{2013}.
\newblock \bibinfo{title}{Polyhedral study of simple plant location problem
  with order}.
\newblock \bibinfo{journal}{Operations Research Letters} \bibinfo{volume}{41},
  \bibinfo{pages}{153--158}.
%Type = Article
\bibitem[{Vasilyev and Klimentova(2010)}]{vasilyev2010}
\bibinfo{author}{Vasilyev, I.L.}, \bibinfo{author}{Klimentova, K.},
  \bibinfo{year}{2010}.
\newblock \bibinfo{title}{The branch and cut method for the facility location
  problem with client’s preferences}.
\newblock \bibinfo{journal}{Journal of Applied and Industrial Mathematics}
  \bibinfo{volume}{4}, \bibinfo{pages}{441--454}.
%Type = Article
\bibitem[{Vasil’ev et~al.(2009)Vasil’ev, Klimentova and
  Kochetov}]{vasil2009}
\bibinfo{author}{Vasil’ev, I.L.}, \bibinfo{author}{Klimentova, K.B.},
  \bibinfo{author}{Kochetov, Y.A.}, \bibinfo{year}{2009}.
\newblock \bibinfo{title}{New lower bounds for the facility location problem
  with clients’ preferences}.
\newblock \bibinfo{journal}{Computational Mathematics and Mathematical Physics}
  \bibinfo{volume}{49}, \bibinfo{pages}{1010--1020}.
%Type = Book
\bibitem[{Wolsey and Nemhauser(1999)}]{wolsey1999integer}
\bibinfo{author}{Wolsey, L.A.}, \bibinfo{author}{Nemhauser, G.L.},
  \bibinfo{year}{1999}.
\newblock \bibinfo{title}{Integer and combinatorial optimization}.
\newblock \bibinfo{publisher}{John Wiley \& Sons}.

\end{thebibliography}

\end{document}